\documentclass[final,leqno,onefignum,onetabnum]{siamltex1213}
\usepackage{amssymb,amsfonts,amsmath,epsfig,hyperref}
\usepackage[margin=1.3in]{geometry}
\usepackage{mathrsfs}
\allowdisplaybreaks
\newcommand{\nn}{\nonumber}
\def\refe#1{(\ref{#1})}
\def\R{\mathbb{R}}
\def\C{\mathbb{C}}

\def\d{\,{\rm d}}

\title{Mathematical and numerical analysis
of time-dependent 
Ginzburg--Landau equations 
in nonconvex polygons 
based on Hodge decomposition
\thanks{This work is supported in part by the National Natural Science Foundation of China
(NSFC) under grants No. 11301262, No. 11471031, No. 91430216,
and the US National Science Foundation (NSF) through grants DMS-1115530 and DMS-1419040.}} 

\author{Buyang Li\thanks{Department of Mathematics,
Nanjing University, Nanjing, 
210093, Jiangsu, P.R. China.
(buyangli@nju.edu.cn)} 
\and Zhimin Zhang\thanks{Beijing Computational
Science Research Center,
Beijing, 100084, P.R. China.\newline
\indent ~$^\ddagger$Department
of Mathematics, Wayne State University, Detroit, MI 48202, USA.
(zzhang@math.wayne.edu)}}

\begin{document}

\maketitle
\slugger{mms}{xxxx}{xx}{x}{x--x}%slugger should be set to mms, siap, sicomp, sicon, sidma, sima, simax, sinum, siopt, sisc, or sirev

\begin{abstract}{\small
We prove well-posedness of time-dependent
Ginzburg--Landau system in a
nonconvex polygonal domain,
and decompose the solution as a
regular part plus a singular part.
We see that the magnetic
potential is not in $H^1$ in general,
and the finite element method
(FEM) may give incorrect solutions.
To remedy this situation,
we reformulate the equations into
an equivalent system of elliptic and parabolic equations
based on the Hodge decomposition, which avoids direct calculation of the magnetic potential.
The essential unknowns of the reformulated system  
admit $H^1$ solutions and
can be solved correctly by the FEMs.
We then propose a decoupled and linearized
FEM to solve the reformulated equations
and present error estimates
based on proved regularity of the solution.
Numerical examples are provided to support our
theoretical analysis and show the efficiency of the method.
}
\end{abstract}

\begin{keywords}{\small
superconductivity, 
reentrant corner, singularity, well-posedness, 
finite element method, convergence,
Hodge decomposition 
}
\end{keywords}

\begin{AMS}{\small
35Q56, 35K61, 65M12, 65M60
}
\end{AMS}

\pagestyle{myheadings}
\thispagestyle{plain}
\markboth{}{}

\section{Introduction}
\setcounter{equation}{0} 

The Ginzburg--Landau theory,
initially introduced by Ginzburg and Landau \cite{GL}
and subsequently extended to the time-dependent case 
by Gor'kov and Eliashberg
\cite{GE}, %is a macroscopic phenomenological model
are widely used to describe the phenomena of superconductivity 
in both low and high temperatures \cite{Gennes,Tinkham}.
In a two-dimensional domain $\Omega\subset\R^2$, 
the time-dependent Ginzburg--Landau model (TDGL) 
is governed by two equations 
(with the Lorentz gauge),   
\begin{align}
&\eta\frac{\partial \psi}{\partial t} 
+ \bigg(\frac{i}{\kappa} \nabla + \mathbf{A}\bigg)^{2} \psi
 + (|\psi|^{2}-1) \psi 
 -i\eta \kappa \psi \nabla\cdot{\bf A} = 0,
\label{PDE1}\\[5pt]
&\frac{\partial \mathbf{A}}{\partial t} 
+ \nabla\times(\nabla\times{\bf A})
-\nabla(\nabla\cdot{\bf A}) 
+  {\rm Re}\bigg[\psi^*\bigg(\frac{i}{\kappa} \nabla 
+ \mathbf{A}\bigg) \psi\bigg] =  \nabla\times f ,
\label{PDE2}
\end{align}
where $\eta$ and $k$ are given positive constants,
the order parameter $\psi$ is an unknown
complex scalar function and $\psi^*$ 
denotes the complex conjugate of $\psi$, the real-vector 
valued
function ${\bf A}=(A_1,A_2)$ denotes the 
unknown magnetic potential, and the scalar function $f$ 
denotes the external magnetic field,
and we have used the notations 
\begin{align*}
&\nabla\times {\bf A}
=\frac{\partial A_2}{\partial x_1}-\frac{\partial A_1}{\partial x_2},
\qquad 
\nabla\cdot {\bf A}=\frac{\partial A_1}{\partial x_1}
+\frac{\partial A_2}{\partial x_2},\\
&\nabla\times f=\bigg(\frac{\partial f}{\partial x_2},\,
-\frac{\partial f}{\partial x_1}\bigg),\quad 
\nabla\psi=\bigg(\frac{\partial \psi}{\partial x_1},\,
\frac{\partial \psi}{\partial x_2}\bigg).
\end{align*}
The natural boundary and initial conditions 
for this problem are  
\begin{align}
& \nabla \psi\cdot{\bf n}= 0,
\quad {\bf A}\cdot{\bf n}=0, 
\quad \nabla\times{\bf A}= f  , 
\quad \mathrm{on}\,\,\,\, \partial \Omega \times (0,T),&
\label{bc}\\
& \psi(x,0) = \psi_{0}(x), \quad \mathbf{A}(x,0) =
\mathbf{A}_{0}(x), \qquad\, \mathrm{in}\,\,\,\,  \Omega \, ,
\label{init}
\end{align}
where $\mathbf{n}$ denotes the unit outward normal vector on the boundary
$\partial\Omega$.

The TDGL has been widely studied 
both theoretically and numerically. 
Existence and uniqueness of the solution 
for \refe{PDE1}-\refe{PDE2} in a smooth domain 
were proved by Chen et al. \cite{CHL},
where equivalence of \refe{PDE1}-\refe{PDE2} to
the Ginzburg--Landau equations under the temporal gauge 
was proved. 
Various numerical methods for solving the 
TDGL were reviewed in \cite{Du92,Du05}. 
In contrast with the many numerical approximation schemes, 
numerical analysis of the model seems very limited so far.
Error analysis of a Galerkin finite element method (FEM) 
with an implicit backward Euler time-stepping scheme was presented in 
\cite{CH95,DuFEM94}, where optimal-order convergence rate 
of the numerical solution was proved for sufficiently regular solution. 
A linearized Crank--Nicolson scheme was proposed in 
\cite{Mu97} for a regularized TDGL under the temporal gauge
without error analysis.
An alternating Crank--Nicolson scheme was proposed
in \cite{MH98} and error estimates were presented for
a regularized TDGL
under the grid-ratio restriction 
$\tau=O(h^{\mbox{\tiny$\frac{11}{12}$}})$,
where $\tau$ and $h$ are the time-step size and spatial mesh size.
Although convergence of the numerical solutions has been 
proved in \cite{CH95,DuFEM94,MH98} in smooth domains, 
these error estimates may not hold in a domain with corners, 
where the regularity of the solution may not satisfy the conditions 
required in the analysis. 
It has been reported in \cite{GLS,Mu97} that 
the numerical solution of the magnetic potential
by the FEM often exhibits undesired singularities around a corner.
%which is not consistent with the exact solution.
To resolve this problem, 
a mixed FEM was proposed in \cite{Chen97} to approximate 
the triple $(\nabla\times {\bf A}$, $\nabla\cdot{\bf A},{\bf A})$ in  
a finite element subspace of 
$H^1(\Omega)\times H^1(\Omega)\times{\bf L}^2(\Omega)$, 
which requires less regularity of ${\bf A}$ intuitively, 
and error estimates of the finite element solution were presented 
under the assumption that ${\bf A}$ is in 
${\bf H}^1_{\rm n}(\Omega):=\{{\bf a}\in H^1(\Omega)^2: 
{\bf a}\cdot{\bf n}=0~\,\mbox{on}~\,\partial\Omega\}$. 
Recently, an optimal-order error estimate 
of the FEM with a linearized Crank--Nicolson 
scheme was presented in \cite{GLS} 
without restriction on the grid ratio, but the analysis requires 
stronger regularity of the solution and the domain.
%No numerical examples were provided in \cite{Chen97}.
On one hand, existing theoretical and numerical analysis of the model 
all require the magnetic potential to be
in ${\bf H}^1_{\rm n}(\Omega)$. 
In a domain with reentrant corners, however,
the magnetic potential may not be in ${\bf H}^1_{\rm n}(\Omega)$
and well-posedness of the TDGL remains open.
%as ${\bf H}^1_{\rm n}(\Omega)$ is no longer equivalent 
%to the correct solution space
%${\bf H}_{\rm n}({\rm curl},{\rm div})$  
%(see Section 2 for its definition). 
On the other hand, numerical approximations of the TDGL 
in domains with reentrant corners are important for physicists 
to study the effects of surface defects in superconductivity  
\cite{ASPM,VMB03}, which are often 
accomplished by solving \refe{PDE1}-\refe{PDE2} directly
with the finite element or finite difference methods,
without being aware of the danger of these numerical methods.

In this paper, we study the TDGL in a nonconvex polygon,
possibly with reentrant corners.
We shall prove that the system {\rm\refe{PDE1}-\refe{init}}
is well-posed,
with ${\bf A}\in L^\infty((0,T);H^s(\Omega)^2)$
for some $s\in(0,1)$ which depends on the interior angles 
of the reentrant corners.
As shown in the numerical examples, with such low-regularity, 
the FEM may give an incorrect solution 
for the magnetic potential ${\bf A}$, which further pollutes the 
numerical solution of $\psi$  
due to the coupling of equations. 
We are interested in reformulating {\rm\refe{PDE1}-\refe{init}} 
into an equivalent form which can be solved correctly 
by the FEMs, as they are preferred when using
software packages and when 
other equations are coupled with the 
Ginzburg--Landau equations. 
%For this purpose, we reformulate the model into
%an equivalent form which can be solved correctly 
%by the $C^0$-FEMs. 
Our idea is to apply the Hodge decomposition
${\bf A} = \nabla\times u + \nabla v$, and consider
the projection of \refe{PDE2} onto the 
divergence-free and curl-free subspaces, 
respectively. Then {\rm\refe{PDE1}-\refe{init}}
is reformulated as
\begin{align}
&\eta\frac{\partial \psi}{\partial t} 
+ \bigg(\frac{i}{\kappa} \nabla + {\bf A}\bigg)^{2} \psi
 + (|\psi|^{2}-1) \psi-i\eta \kappa \psi \nabla\cdot{\bf A} = 0,
\label{RFPDE1}\\[5pt]
&\Delta p =-\nabla\times \bigg({\rm Re}\bigg[
\psi^*\bigg(\frac{i}{\kappa} \nabla 
+ \mathbf{A}\bigg) \psi\bigg]\bigg)\label{RFPDEp}\\[5pt]
&\Delta q =\nabla\cdot \bigg({\rm Re}\bigg[
\psi^*\big(\frac{i}{\kappa} \nabla 
+ \mathbf{A}\bigg) \psi\bigg]\bigg)\label{RFPDEq}\\[5pt]
&\frac{\partial u}{\partial t} -\Delta u
=  f-p ,\label{RFPDEu}\\[5pt]
&\frac{\partial v}{\partial t} -\Delta v
=  -q ,
\label{RFPDEv}
\end{align}
with the boundary and initial conditions
\begin{align}
&\nabla \psi\cdot{\bf n}= 0,
\quad p=0,\quad \nabla q\cdot{\bf n}=0,\quad
u =0 , \quad \nabla v\cdot{\bf n}=0, 
\quad \mathrm{on}\  \partial \Omega \times (0,T],&
\label{RFbc}\\
& \psi(x,0) = \psi_{0}(x), \quad u(x,0) =
u_{0}(x), \quad v(x,0) =
v_{0}(x),\quad \mathrm{in}\  \Omega \, ,
\label{RFinit}
\end{align}
where $\nabla \times p$ and $\nabla q$ are just
the divergence-free and curl-free parts of 
${\rm Re}\big[\psi^*\big(\frac{i}{\kappa} \nabla 
+ \mathbf{A}\big) \psi\big]$,
respectively, 
$u_0$ and $v_0$ are defined by
$$
\left\{\begin{array}{ll}
-\Delta u_0=\nabla\times{\bf A}_0 &\mbox{in}~~\Omega,\\
u_0=0 &\mbox{on}~~\partial\Omega,
\end{array}\right.
\qquad\mbox{and}\qquad
\left\{\begin{array}{ll}
\Delta v_0=\nabla\cdot{\bf A}_0 &\mbox{in}~~\Omega,\\
\partial_nv_0=0 &\mbox{on}~~\partial\Omega,
\end{array}\right.
$$
with $\int_\Omega v_0(x)\d x=0$. 
We shall prove that the solution of 
the projected TDGL \refe{RFPDE1}-\refe{RFinit} 
coincides with the solution of {\rm\refe{PDE1}-\refe{init}}.
Then we propose a decoupled and linearized FEM 
to solve \refe{RFPDE1}-\refe{RFinit}, 
and establish error estimates based on proved regularity of the solution.
Our main results are presented in Section 2, and we prove these 
results in Section 3-5. In Section 6, we present numerical
examples to support our theoretical analysis.
Due to limitations on pages, 
derivations of the system \refe{RFPDE1}-\refe{RFinit} 
are presented in a separate paper \cite{LZ2},
where the efficiency of the method is shown 
via numerical simulations in comparison with
the traditional approaches of solving 
the TDGL directly under the temporal gauge 
and the Lorentz gauge.

%We would like to mention that the Hodge decomposition 
%has been used to solve static electromagnetic problems before; 
%see \cite{AVV09,AM10,BCNS}. 
%%In particular, we are inspired by the work of Brenner et. al. \cite{BCNS}.
%Our method is based on the projection of 
%the time-dependent magnetic potential equation 
%onto the divergence-free and 
%curl-free subspaces, 
%which is different from these previous works.
%\medskip

\section{Main results}\label{femmethod}
\setcounter{equation}{0}

For any nonnegative integer $k$, we 
let $W^{k,p}(\Omega)$,  and ${\mathcal W}^{k,p}(\Omega)$ denote the the 
conventional Sobolev spaces
of real-valued and complex-valued 
functions defined in $\Omega$, respectively,
with $H^{k}(\Omega)=W^{k,2}(\Omega)$, 
${\mathcal  H}^{k}(\Omega)={\mathcal  W}^{k,2}(\Omega)$,
$L^2(\Omega)=H^0(\Omega)$ and
${\mathcal  L}^2(\Omega)={\mathcal  H}^0(\Omega)$;
see \cite{Adams}. 
For a positive real number $s_0=k+s$,
with $s\in (0,1)$, we define
$H^{s_0}(\Omega)=(H^{k}(\Omega),H^{k+1}(\Omega))_{[s]}$ via
the complex interpolation; see \cite{BL76}.
We denote $H^{s_0}=H^{s_0}(\Omega)$,
${\mathcal  H}^{s_0}={\mathcal  H}^{s_0}(\Omega)$, 
$L^p=L^p(\Omega)$, 
${\mathcal  L}^p={\mathcal  L}^p(\Omega)$, 
and let $\mathring H^1$ denote the subspace
of $H^1$ consisting of functions whose 
traces are zero on $\partial\Omega$.
For any two functions $f,g\in {\mathcal  L}^2$ we define
$$
(f,g)=\int_\Omega f(x)g(x)^*\d x ,
$$
where $g(x)^*$ denotes the complex conjugate of $g(x)$,
and define
\begin{align*}
&{\bf L}^p=L^p\times L^p ,
\qquad
{\bf H}^s=H^s \times H^s ,
\qquad
{\bf H}^1_{\rm n}(\Omega):=\{{\bf a}\in H^1\times H^1: 
{\bf a}\cdot{\bf n}=0~\,\mbox{on}~\,\partial\Omega\},\\
&{\bf H}_{\rm n}({\rm curl},{\rm div})=\{{\bf a}\in {\bf L}^2: \nabla\times {\bf a}\in L^2,
~\nabla\cdot{\bf a}\in L^2~\mbox{and}
~{\bf a}\cdot{\bf n}=0~\mbox{on}~\partial\Omega\},\\
&H({\rm curl})=\{g\in L^2: \nabla\times g\in {\bf L}^2\} .
\end{align*}

\begin{definition}\label{DefWSol}
{\bf(Weak solutions of \refe{PDE1}-\refe{init})}~~
{\it Let $\omega$ denote the maximal 
interior angle of the nonconvex polygon $\Omega$.
The pair $(\psi,{\bf A})$ is called a weak solution of 
{\rm\refe{PDE1}-\refe{init}} if
\begin{align*}
&\psi\in C([0,T];{\mathcal  L}^2)\cap 
L^\infty((0,T);{\mathcal  H}^1)\cap L^2((0,T);{\mathcal  H}^{1+s}),\\
&
\partial_t\psi,\Delta\psi\in L^2((0,T);{\mathcal  L}^2) ,
\quad |\psi|\leq 1~~\mbox{a.e.~in~\,}\Omega\times(0,T),\\
& {\bf A}\in C([0,T];{\bf L}^2)\cap 
L^{\infty}((0,T); {\bf H}_{\rm n}({\rm curl},{\rm div})) , \\
&
\partial_t{\bf A}\in L^2((0,T);{\bf L}^2),\quad
\nabla\times {\bf A},\nabla\cdot{\bf A}\in L^2((0,T);H^1),
\end{align*}
for any $s\in(1/2,\pi/\omega)$, with $\psi(\cdot,0)=\psi_0 $, 
${\bf A}(\cdot,0)={\bf A}_0$, 
and the variational equations 
\begin{align}
&\int_0^T\bigg[\bigg(\eta\frac{\partial \psi}{\partial t} ,\varphi\bigg)
+ \bigg(\bigg(\frac{i}{\kappa} \nabla + \mathbf{A}\bigg)  \psi,
\bigg(\frac{i}{\kappa} \nabla + \mathbf{A}\bigg)\varphi\bigg)\bigg]\d t\nn\\
&\qquad\qquad\quad +\int_0^T\bigg[\bigg( (|\psi|^{2}-1) \psi
 -i\eta \kappa \psi \nabla\cdot{\bf A},\varphi\bigg)\bigg]\d t = 0,
\label{VPDE1}\\[10pt]
&\int_0^T\bigg[\bigg(\frac{\partial \mathbf{A}}{\partial t} ,{\bf a}\bigg)
+ \big(\nabla\times{\bf A},\nabla\times{\bf a}\big)
+\big(\nabla\cdot{\bf A},\nabla\cdot{\bf a}\big) \bigg]\d t   \nn\\
&=
 \int_0^T\bigg[\big(f , \nabla\times {\bf a}\big)
 -\bigg( {\rm Re}\bigg[\psi^*\bigg(\frac{i}{\kappa} \nabla 
+ \mathbf{A}\bigg) \psi\bigg],{\bf a}\bigg) \bigg]\d t,
\label{VPDE2}
\end{align}
hold for all $\varphi\in L^2((0,T);{\mathcal  H}^1)$ 
and ${\bf a}\in L^2((0,T);{\bf H}_{\rm n}({\rm curl},{\rm div}))$.
}
\end{definition}\medskip

\begin{definition}\label{DefWSol2}
{\bf(Weak solutions of \refe{RFPDE1}-\refe{RFinit})}~~
{\it Let $\omega$ denote the maximal 
interior angle of the nonconvex polygon $\Omega$.
The quintuple $(\psi,p,q,u,v)$ is called a weak solution of 
{\rm\refe{RFPDE1}-\refe{RFinit}} if
\begin{align*}
&\psi\in C([0,T];{\mathcal  L}^2)\cap 
L^\infty((0,T);{\mathcal  H}^1)\cap L^2((0,T);{\mathcal  H}^{1+s}),\\
&
\partial_t\psi,\Delta\psi\in L^2((0,T);{\mathcal  L}^2) ,
\quad |\psi|\leq 1~~\mbox{a.e.~in~\,}\Omega\times(0,T),\\
&p\in L^\infty((0,T);\mathring H^1),\quad q\in L^\infty((0,T);H^1),\quad
u \in C([0,T];\mathring H^1) , \quad  v\in C([0,T];H^1),\\
& \partial_tu,\partial_tv,\Delta u,\Delta v
\in L^\infty((0,T);L^2)\cap L^2((0,T);H^1), 
\end{align*}
for any $s\in(1/2,\pi/\omega)$, with $\psi(\cdot,0)=\psi_0 $, 
$u(\cdot,0)=u_0$, $v(\cdot,0)=v_0$, 
and the variational equations 
\begin{align}
&\int_0^T\bigg[\bigg(\eta\frac{\partial \psi}{\partial t} ,\varphi\bigg)
+ \bigg(\bigg(\frac{i}{\kappa} \nabla + \mathbf{A}\bigg)  \psi,
\bigg(\frac{i}{\kappa} \nabla + \mathbf{A}\bigg)\varphi\bigg)\bigg]\d t\nn\\
&\qquad\qquad\quad +\int_0^T\bigg[\bigg( (|\psi|^{2}-1) \psi
-i\eta \kappa \psi \nabla\cdot{\bf A},
 \varphi\bigg) \bigg]\d t= 0,
\label{VPDE1-2}\\ 
&\int_0^T\big(\nabla p,\nabla\xi\big)\d t
=\int_0^T \bigg({\rm Re}\Big[\psi^*\Big(\frac{i}{\kappa} \nabla
 + \mathbf{A}\Big) \psi\Big],\nabla\times\xi\bigg)\d t\label{RFPDEp-2}\\[5pt]
&\int_0^T\big(\nabla q,\nabla\zeta\big) \d t=
\int_0^T\bigg({\rm Re}\Big[\psi^*\Big(\frac{i}{\kappa} \nabla 
+ \mathbf{A}\Big) \psi\Big],\nabla\zeta\bigg)\d t\label{RFPDEq-2}\\[5pt]
&\int_0^T\bigg[\bigg(\frac{\partial u}{\partial t},\theta\bigg) 
+ \big(\nabla u,\nabla\theta\big)\bigg]\d t
= \int_0^T\big( f-p,\theta\big) \d t,\label{RFPDEu-2}\\[5pt]
&\int_0^T\bigg[\bigg(\frac{\partial v}{\partial t},\vartheta\bigg) 
+\big(\nabla v,\nabla\vartheta\big)\bigg]\d t
=  -\int_0^T\big(q,\vartheta\big) \d t,
\label{RFPDEv-2} 
\end{align}
hold for all $\varphi\in L^2((0,T);{\mathcal  H}^1)$, 
$\xi,\theta\in L^2((0,T);\mathring H^1)$ and $\zeta,\vartheta\in L^2((0,T);H^1)$. 
}
\end{definition}

\setcounter{theorem}{0}
Our first result is the well-posedness 
and equivalence of the systems \refe{PDE1}-\refe{init} and
\refe{RFPDE1}-\refe{RFinit}, which are presented in the following theorem.\medskip

\begin{theorem}\label{MainTHM1}
{\bf(Well-posedness and equivalence of the two systems)}$~$\\
{\it
If $f\in L^\infty((0,T);L^2)\cap L^2((0,T);H({\rm curl}))$, 
$\psi_0\in {\mathcal  H}^1$,
${\bf A}_0\in {\bf H}_{\rm n}({\rm curl},{\rm div})$ and 
$|\psi_0|\leq 1$ a.e. in $\Omega$, then 
the system {\rm\refe{PDE1}-\refe{init}} 
admits a unique weak solution
in the sense of Definition {\rm\ref{DefWSol}},
and the system {\rm\refe{RFPDE1}-\refe{RFinit}} 
admits a unique solution  which coincides 
with the solution of {\rm\refe{PDE1}-\refe{init}}.

Moreover, if we let $x_j$, $j=1,\cdots,m$, 
be the reentrant corners of the 
domain $\Omega$, then the solution has the decomposition
\begin{align*}
&\psi(x,t)=\Psi(x,t)
+\sum_{j=1}^m\alpha_j(t)\Phi(|x-x_j|)|x-x_j|^{\pi/\omega_j}
\cos(\pi\Theta_j(x)/\omega_j) , \\[3pt]
&{\bf A}=\nabla\times u+\nabla v
\end{align*}
with
\begin{align*}
&u(x,t)=\widetilde u(x,t)+
\sum_{j=1}^m\beta_j(t)\Phi(|x-x_j|)|x-x_j|^{\pi/\omega_j}
\sin(\pi\Theta_j(x)/\omega_j),\\
&v(x,t)=\widetilde v(x,t)
+\sum_{j=1}^m\gamma_j(t)\Phi(|x-x_j|)|x-x_j|^{\pi/\omega_j}
\cos(\pi\Theta_j(x)/\omega_j),
\end{align*}
where $\Psi\in L^2((0,T);{\mathcal  H}^2)$, 
$\widetilde u,\widetilde v\in L^\infty((0,T);H^2)$, 
$\Phi(r)$ is a given smooth cut-off function which equals 
$1$ in a neighborhood of $0$,
$\Theta_j(x)$ is the angle shown in Figure \ref{LshapeD},
and $\alpha_j,\beta_j,\gamma_j\in L^2(0,T)$.
}
\end{theorem}
\begin{figure}[ht]
\vspace{0.1in}
\centering
\includegraphics{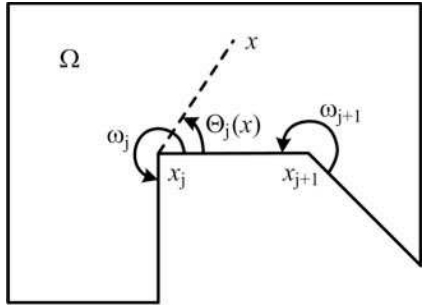}
\caption{\small Illustration of the domain $\Omega$, corner $x_j$,
angle $\omega_j$ and $\Theta_j(x)$.}
\label{LshapeD}
\end{figure}

Further regularity of the solution is presented below, 
which is needed in the analysis of the convergence 
of the numerical solution.\medskip

\begin{theorem}\label{MainTHM2}
{\bf(Further regularity)}$~$\\
{\it
If $f\in C([0,T];H({\rm curl}))$, $\nabla\times f \in L^2((0,T);H({\rm curl}))$,
$\partial_tf \in L^2((0,T);L^2)$, $\psi_0\in {\mathcal  H}^1$, $\Delta\psi_0\in {\mathcal  L}^2$, 
${\bf A}_0\in {\bf H}_{\rm n}({\rm curl},{\rm div})$, 
$\nabla\cdot{\bf A}_0,\nabla\times{\bf A}_0\in {\bf H}^1$,
$|\psi_0|\leq 1$ a.e. in $\Omega$,
and the compatibility conditions
$$
\partial_n\psi_0=0\quad\mbox{and}\quad \nabla\times{\bf A}_0=f(\cdot,0)
\quad\mbox{on}\quad \partial\Omega 
$$ 
are satisfied,
then the solution of {\rm\refe{RFPDE1}-\refe{RFinit}}
possesses the regularity
\begin{align*}
&\psi\in C([0,T];{\mathcal  H}^{1+s}),
\quad 
\partial_{t}\psi
\in L^2((0,T);{\mathcal  H}^{1+s}),
\quad 
\partial_{tt}\psi\in L^2((0,T);{\mathcal  L}^2), \\
&p,q\in L^\infty((0,T);H^{1}),\quad u,v\in C([0,T];H^{1+s}),\\
& \partial_{t}u,\partial_{t}v
\in L^2((0,T);H^{1+s}),
\quad 
\partial_{tt}u,\partial_{tt}v\in L^2((0,T);L^2)
\end{align*}
for any $s\in(1/2,\pi/\omega)$.
}
\end{theorem}

To solve the reformulated system \refe{RFPDE1}-\refe{RFinit}, 
we propose a decoupled and linearized Galerkin
FEM. For this purpose, we let $\pi_h$ be a quasi-uniform 
triangulation of the domain $\Omega$
and denote the mesh size by $h$. Let
$\mathcal{V}_{h}^{1}$ denote the space 
of complex-valued $C^0$ 
piecewise linear functions subject to the triangulation, 
let $V_{h}^{1}$ denote the space of real-valued
$C^0$ piecewise linear functions, 
and set $ \mathring V_{h}^{1}= \{\varphi\in V_{h}^{1}:
\varphi=0~~\mbox{on}~~\partial\Omega\}$. 
Clearly, $\mathcal{V}_{h}^{1}$, 
$\mathring V_{h}^{1}$ and $V_{h}^{1}$ are
finite dimensional subspaces of $\mathcal{H}^{1}$, $\mathring H^1$
and $H^1$, respectively. Let $I_{h}$ be
the commonly used Lagrange interpolation operator onto the
finite element spaces. For any positive integer $N$, 
we let $0=t_0<t_1<\cdots<t_N=T$
be a uniform partition of the time interval $[0,T]$ and
set $\tau = T/N$. For any sequence of functions $\varphi^n$, we define 
${D_{\tau}} \varphi^{n+1} := (\varphi^{n+1}-\varphi^{n})/\tau$,
and we define a cut-off function $\chi:\C\rightarrow\C$ by
$$
\chi(z)=z/\max(|z|,1) ,\quad\forall~ z\in\C ,
$$
which is Lipschitz continuous and satisfies that
$|\chi(z)|\leq 1$, $\forall\, z\in\C$.
%and possesses Lipschitz continuity:
%$$|\chi(z_1)-\chi(z_2)|\leq |z_1-z_2| ,\quad\forall\, z_1,z_2\in\C .$$

We look for $\psi^{n+1}_h\in {\mathcal  V}_{h}^{1}$, 
$p^{n+1}_h,u^{n+1}_h\in
\mathring V_h^{1}$ and
$q^{n+1}_h,v^{n+1}_h\in
V_h^{1}$ satisfying the equations
\begin{align}
&\big(D_\tau \psi^{n+1}_h, \varphi\big) +\big
((i\kappa^{-1}\nabla + \mathbf{A}^{n}_h)
\psi^{n+1}_h,(i\kappa^{-1}\nabla + \mathbf{
A}^{n}_h)  \varphi\big) \nn\\
&\qquad +\big
((|\psi^{n}_h|^{2}-1) \psi^{n+1}_h,
\varphi\big)
+\big(i\eta\kappa {\bf A}_h^n,\nabla((\psi_h^{n+1})^* \varphi\big)= 0,
\label{FEMEq1}\\
&(\nabla p^{n+1}_h,\nabla \xi)=\big({\rm Re}
[\chi(\psi^n_h)^*(i\kappa^{-1}\nabla\psi^{n+1}_h 
+ \mathbf{A}^n_h \psi^{n+1}_h)],
\nabla\times \xi\big)
\label{FEMEq2}\\
&(\nabla q^{n+1}_h,\nabla \zeta)
=\big({\rm Re}
[\chi(\psi^n_h)^*(i\kappa^{-1}\nabla\psi^{n+1}_h 
+ \mathbf{A}^n_h \psi^{n+1}_h)],\nabla \zeta\big)
\label{FEMEq3}\\
&\big(D_\tau u^{n+1}_h, \theta\big) 
+\big(\nabla u^{n+1}_h,\nabla\theta\big)=(f^{n+1}-p^{n+1}_h,\theta)\label{FEMEq4}\\
&\big(D_\tau v^{n+1}_h, \vartheta\big) 
+\big(\nabla v^{n+1}_h,\nabla\vartheta\big)
=(-q^{n+1}_h,\vartheta),\label{FEMEq5}
\end{align}
for all $\varphi\in{\mathcal  V}^1_h$, 
$\xi,\theta\in \mathring V^1_h$ and $\zeta,\vartheta\in V^1_h$, 
with ${\bf A}_h^n=\nabla\times u_h^n+\nabla v_h^n$, 
where $u^0_h\in\mathring V^1_h$ 
and $v^0_h\in V^1_h$ are solved from
\begin{align}
&(\nabla u^{0}_h,\nabla \xi)=\big({\bf A},
\nabla\times \xi\big)
,\quad\forall~\xi\in\mathring V_h^1 ,
\label{FEMEqu0}\\
&(\nabla v^{0}_h,\nabla \zeta)
=\big({\bf A},\nabla\cdot \zeta\big)
,\quad~\, \forall~\zeta\in V_h^1 ,
\label{FEMEqv0}
\end{align}
and $\psi^0_h$ is the Lagrange interpolation of $\psi^0$.

For the proposed scheme, we have the following theorem
concerning the convergence of the numerical solution.\medskip

\begin{theorem}\label{MainTHM3}
{\bf(Convergence of the finite element solution)}~\\
{\it The finite
element system {\rm(\ref{FEMEq1})-(\ref{FEMEq5})} admits a unique
solution $(\psi_h^{n},p_h^{n},q_h^{n},
u_h^{n},v_h^{n})$ when $\tau<\eta/4$ and, 
under the assumptions of Theorem {\rm\ref{MainTHM2}},   
\begin{align*}
&\max_{1\leq n\leq N}\big( \|u^n-u^n_h\|_{H^1} 
+ \|v^n-v^n_h\|_{H^1} 
+ \|{\bf A}^n-{\bf A}_h^n\|_{L^2}
+ \|\psi^n-\psi^n_h\|_{{\mathcal  L}^2} \big)\leq C(\tau+h^s) ,
\end{align*}
where $C$ is a positive constant independent of $\tau$ and $h$.
}
\end{theorem}

In the rest part of this paper, we prove Theorem 
\ref{MainTHM1}--\ref{MainTHM3}. 
To simplify the notations, we denote by $C$ a generic positive
constant which may be different at each occurrence but
is independent of $n$, $\tau$ and $h$.\medskip

\section{Proof of Theorem \ref{MainTHM1}}\label{WPness}
\setcounter{equation}{0}

In this section, we prove well-posedness of the 
Ginzburg--Landau equations in a nonconvex polygon 
and equivalence of the two formulations 
\refe{PDE1}-\refe{init} and
\refe{RFPDE1}-\refe{RFinit}.
Compared with smooth domains, 
in a nonconvex polygon, 
the space ${\bf H}_{\rm n}({\rm curl},{\rm div})$ is not equivalent 
to ${\bf H}^1_{\rm n}(\Omega)$ and is not embedded
into ${\bf L}^p$ for large $p$. 
Convergence of the nonlinear terms 
of the approximating solutions 
needs to be proved based on the weaker embedding 
${\bf H}_{\rm n}({\rm curl},{\rm div})
\hookrightarrow\hookrightarrow {\bf L}^4$
in the compactness argument,
and uniqueness of solution needs to 
be proved based on weaker
regularity of the solution.

\subsection{Preliminaries}
 
Firstly, we cite a lemma concerning the regularity 
of Poisson's equations
in a nonconvex polygon \cite{Dauge,Grisvard}.
\begin{lemma}\label{HsPoissEq}
{\it The solution of the Poisson equations
$$
\left\{\begin{array}{ll}
\Delta w=g &\mbox{in}~~\Omega,\\
w=0 &\mbox{on}~~\partial\Omega,
\end{array}\right.
\qquad\mbox{and}\qquad
\left\{\begin{array}{ll}
\Delta w=g &\mbox{in}~~\Omega,\\
\partial_nw=0 &\mbox{on}~~\partial\Omega,
\end{array}\right.
$$
satisfies that $($the Neumann problem requires 
$\int_\Omega g(x)\d x=\int_\Omega w(x)\d x=0$$)$
$$
\|w\|_{W^{1,p_s}}+\|w\|_{H^{1+s}}\leq C_s\|g\|_{L^2},
\quad \forall~s\in(1/2,\pi/\omega) ,
$$
where $p_s=2/(1-s)>4$ when $s\in(1/2,\pi/\omega)$.
}
\end{lemma}

Secondly, we introduce a lemma concerning 
the embedding of ${\bf H}_{\rm n}({\rm curl},{\rm div})$
into ${\bf H}^s$.

\begin{lemma}\label{LemCDReg} {\it 
${\bf H}_{\rm n}({\rm curl},{\rm div})
\hookrightarrow {\bf H}^s\hookrightarrow {\bf L}^{p_s}$ 
for any $s\in(1/2,\pi/\omega)$.
}
\end{lemma}
\begin{proof}$~$ 
From \cite{Chen97} we know that 
${\bf A}$ has the decomposition
${\bf A}=\nabla\times u+\nabla v$, where 
$u$ and $v$ are the solutions of 
$$
\left\{\begin{array}{ll}
-\Delta u=\nabla\times{\bf A} &\mbox{in}~~\Omega,\\
u=0 &\mbox{on}~~\partial\Omega,
\end{array}\right.
\qquad\mbox{and}\qquad
\left\{\begin{array}{ll}
\Delta v=\nabla\cdot{\bf A} &\mbox{in}~~\Omega,\\
\partial_nv=0 &\mbox{on}~~\partial\Omega,
\end{array}\right.
$$
respectively, with $\int_\Omega v(x)\d x=0$.
For the two Poisson's equations, Lemma \ref{HsPoissEq}
implies that
\begin{align*}
&\|u\|_{H^{1+s}}+\|v\|_{H^{1+s}}  
\leq C_s(\|\nabla\times{\bf A}\|_{L^2}+\|\nabla\cdot{\bf A}\|_{L^2}),
\quad\,\forall~s\in(1/2,\pi/\omega) .
\end{align*}
\end{proof} 

Thirdly, we introduce a lemma concerning the embedding 
of discrete Sobolev spaces.

\begin{lemma}\label{LemdisEmb} {\it 
Let $\theta_h\in \mathring V_h^1$, $\vartheta_h\in V_h^1$ 
with $\int_\Omega \vartheta_h(x)\d x=0$. If we define 
$\Delta_h\theta_h\in\mathring V_h^1$ 
and $\Delta_h\vartheta_h\in V_h^1$ by
\begin{align*}
&(\Delta_h\theta_h,\varphi)=-(\nabla \theta_h,\nabla\varphi),
\quad\forall\varphi\in \mathring V_h^1,\\
&(\Delta_h\vartheta_h,\varphi)=-(\nabla\vartheta_h,\nabla\varphi),
\quad\forall\varphi\in V_h^1,
\end{align*}
then 
\begin{align*}
&\|\nabla \theta_h\|_{L^4}\leq C\|\Delta_h\theta_h\|_{L^2},
\qquad\mbox{and}\qquad
\|\nabla \vartheta_h\|_{L^4}\leq C\|\Delta_h\vartheta_h\|_{L^2} .
\end{align*}
}
\end{lemma}
\begin{proof}$~$
Let $\theta$ be the solution of the Poisson's equation
\begin{align*}
&\Delta\theta=\Delta_h\theta_h
\end{align*}
with the Dirichlet boundary condition 
$\theta=0$ on $\partial\Omega$.
Then $(\nabla(\theta-\theta_h),\nabla\xi_h)=0$ 
for any $\xi_h\in \mathring V_h^1$,
which implies that, via the standard $H^1$-norm error estimate
and Lemma \ref{HsPoissEq},
\begin{align*}
\|\nabla(\theta-\theta_h)\|_{L^2}\leq C\|\theta\|_{H^{1+s}}h^s 
\leq C\|\Delta_h\theta_h\|_{L^2}h^s .
\end{align*}
Since $s>1/2$, by applying the inverse inequality we obtain that
\begin{align*}
\|\nabla(I_h\theta-\theta_h)\|_{L^4}
\leq Ch^{-1/2}\|\nabla(I_h\theta-\theta_h)\|_{L^2} 
\leq C\|\Delta_h\theta_h\|_{L^2}h^{s-1/2} 
\leq C\|\Delta_h\theta_h\|_{L^2}.
\end{align*}
Thus $\|\nabla \theta_h \|_{L^4}
\leq \|\nabla(I_h\theta-\theta_h)\|_{L^4}
+\|\nabla I_h \theta \|_{L^4}\leq C\|\Delta_h\theta_h\|_{L^2} $.
The proof for $\vartheta_h$ is similar.
\end{proof}\medskip

\subsection{Existence of weak solutions for \refe{PDE1}-\refe{init}}
In this subsection, we prove existence of weak solutions for 
the system \refe{PDE1}-\refe{init} by constructing approximating
solutions in finite dimensional spaces and then applying a 
compactness argument. Firstly, we need the following lemma
to control the order parameter pointwisely.
\begin{lemma}\label{UnBDPsi}
{\it For any given ${\bf A}\in 
L^\infty((0,T);{\bf H}_{\rm n}({\rm curl},{\rm div}))$,  
the equation {\rm\refe{PDE1}}
has at most one weak solution 
$\psi\in L^\infty((0,T);H^1)\cap H^1((0,T);L^2)$ in the 
sense of {\rm\refe{VPDE1}}.
If the solution exists then it 
satisfies that $|\psi|\leq 1$ a.e. in $\Omega\times(0,T)$.
}
\end{lemma}
\begin{proof}$~$
From Lemma \ref{LemCDReg} we see that
${\bf H}_{\rm n}({\rm curl},{\rm div})\hookrightarrow {\bf H}^s 
\hookrightarrow {\bf L}^4 $ and so
${\bf A}\in L^\infty((0,T);{\bf L}^4)$.
Uniqueness of the solution can be proved easily based on
the regularity assumption of $\psi$. 
To prove $|\psi|\leq 1$ a.e. in $\Omega\times(0,T)$,
we integrate \refe{PDE1} against $\psi^*(|\psi|^2-1)_+$ and
consider the real part,
where $(|\psi|^2-1)_+$ denotes
the positive part of $|\psi|^2-1$. For any $t'\in(0,T)$ we have
\begin{align*}
& \int_\Omega
\bigg(\frac{\eta}{4}\big(|\psi(x,t')|^2-1\big)_+ ^2\bigg)\d x
 + \int_0^{t'}\int_\Omega (|\psi|^{2}-1)^2_+ |\psi| ^2\d x\d t\\
&=-\int_0^{t'}{\rm Re}\int_\Omega 
\bigg(\frac{i}{\kappa} \nabla  \psi+ \mathbf{A} \psi\bigg)
\bigg(-\frac{i}{\kappa} \nabla 
+ \mathbf{A}\bigg)[\psi^* (|\psi|^2-1)_+]\d x\d t\\
&=-\int_0^{t'}\int_\Omega \bigg|\frac{i}{\kappa} 
\nabla  \psi+ \mathbf{A} \psi\bigg|^2
 (|\psi|^2-1)_+ \d x \d t\\
&\quad + \int_0^{t'}{\rm Re}\int_{\{|\psi|^2>1\}} 
\bigg(\frac{i}{\kappa} \nabla  \psi
+ \mathbf{A} \psi\bigg)\psi^*\bigg(\frac{i}{\kappa}  
\psi\nabla\psi^*+\frac{i}{\kappa}\psi^*\nabla\psi \bigg)\d x\d t\\
&=-\int_0^{t'}\int_\Omega \bigg|\frac{i}{\kappa} 
\nabla  \psi+ \mathbf{A} \psi\bigg|^2
 (|\psi|^2-1)_+ \d x\d t\\
&\quad -\int_0^{t'}{\rm Re}\int_{\{|\psi|^2>1\}}(|\psi|^2|\nabla\psi|^2
+ (\psi^*)^2\nabla\psi\cdot \nabla\psi )\d x\d t\\
& \leq 0,
\end{align*}
which implies that $\int_\Omega(|\psi(x,t')|^2-1)_+ ^2 \d x
=0$. Thus $|\psi|\leq 1$ a.e. in $\Omega\times(0,T)$. 
\end{proof}\medskip

Secondly, we construct approximating solutions in 
finite dimensional spaces. For this purpose, 
we let $\phi_1,\phi_2,\cdots$ be the 
eigenfunctions of the Neumann Laplacian,
which form a basis of ${\mathcal  H}^1$. 
Let $M:{\bf H}_{\rm n}({\rm curl},{\rm div})
\rightarrow ({\bf H}_{\rm n}({\rm curl},{\rm div}))'$ be defined by
$$
(M{\bf u},{\bf v})=(\nabla\times {\bf u},\nabla\times {\bf v})
+(\nabla\cdot{\bf u},\nabla\cdot{\bf v}),
\quad\mbox{for}~ {\bf u},{\bf v}
\in {\bf H}_{\rm n}({\rm curl},{\rm div}) \, .
$$
Since the bilinear form on the right-hand side 
is coercive on the space ${\bf H}_{\rm n}({\rm curl},{\rm div})$, 
which is compactly embedded into ${\bf L}^2$, 
the spectrum of $M$ consists of a sequence of 
eigenvalues which tend to infinity, and the 
corresponding eigenvectors ${\bf a}_1,{\bf a}_2,{\bf a}_3,\cdots$ 
form a basis of ${\bf H}_{\rm n}({\rm curl},{\rm div})$ 
\cite{CD99,McLean00}.

We define ${\mathcal V}_N
={\rm span}\{\phi_1,\phi_2,\cdots,\phi_N\}$
and ${\bf X}_N={\rm span}\{{\bf a}_1,{\bf a}_2,\cdots,{\bf a}_N\}$,
which are finite dimensional subspaces of 
${\mathcal  H}^1$ and ${\bf H}_{\rm n}({\rm curl},{\rm div})$,
respectively,
and we look for $\Psi_N(t)\in {\mathcal V}_N$,
${\bf \Lambda}_N(t)\in {\bf X}_N$
such that
\begin{align}
&\bigg(\eta\frac{\partial \Psi_N}{\partial t} ,\varphi\bigg)
+ \bigg(\bigg(\frac{i}{\kappa} \nabla +{\bf \Lambda}_N\bigg)  \Psi_N,
\bigg(\frac{i}{\kappa} \nabla + {\bf \Lambda}_N\bigg)\varphi\bigg)\nn\\
&\qquad\quad +\bigg( (|\Psi_N|^{2}-1) \Psi_N-i\eta\kappa
\chi(\Psi_N)\nabla\cdot{\bf\Lambda}_N,\varphi\bigg) = 0,
\label{DVPDE1}\\[8pt]
&\bigg(\frac{\partial {\bf \Lambda}_N}{\partial t} ,{\bf a}\bigg)
+ \big(\nabla\times{\bf \Lambda}_N,\nabla\times{\bf a}\big)
+\big(\nabla\cdot{\bf \Lambda}_N,\nabla\cdot{\bf a}\big)  \nn\\
&\qquad\quad 
+ \bigg( {\rm Re}\bigg[\Psi_N^*\bigg(\frac{i}{\kappa} \nabla 
+{\bf \Lambda}_N\bigg) \Psi_N\bigg],{\bf a}\bigg) =
 \big(f , \nabla\times {\bf a}\big),
\label{DVPDE2}
\end{align}
for any $\varphi\in {\mathcal V}_N$ and ${\bf a}\in {\bf X}_N$ 
at any $t\in (0,T)$,
with the initial conditions $\Psi(0) = \Pi_N\psi_{0}$ and  
${\bf\Lambda}(0) =\widetilde\Pi_N{\bf A}_{0}$, where 
$\Pi_N$ and $\widetilde\Pi_N$ are the 
projections of ${\mathcal  H}^1$ and 
${\bf H}_{\rm n}({\rm curl},{\rm div})$ onto the subspaces 
${\mathcal V}_N$ and ${\bf X}_N$,  respectively.

Existence and uniqueness of solutions for the 
ODE problem \refe{DVPDE1}-\refe{DVPDE2} are obvious. 
To present estimates of the semi-discrete solution 
$(\Psi_N,{\bf \Lambda}_N)$, 
we substitute $\varphi=\partial_t\Psi$ and 
${\bf a}=\partial_t{\bf\Lambda}$ into the equations,
and sum up the two results. Then 
we obtain that
\begin{align*}
&\frac{\d}{\d t}\int_\Omega\frac{1}{2}\bigg(
\bigg|\frac{i}{\kappa}\nabla\Psi_N+{\bf \Lambda}_N\Psi_N\bigg|^2
+\frac{1}{2}(|\Psi_N|^2-1)^2
+|\nabla\times{\bf \Lambda}_N-f|^2
+|\nabla\cdot{\bf \Lambda}_N|^2 \bigg)\d x\\
&\quad +\int_\Omega\bigg(\bigg|
\frac{\partial {\bf \Lambda}_N}{\partial t}\bigg|^2
+\eta\bigg|\frac{\partial \Psi_N}{\partial t}\bigg|^2\bigg)\d x \\
&=\eta\kappa\int_\Omega {\rm Im}\bigg(\chi(\Psi_N)
\frac{\partial \Psi_N^*}{\partial t}\bigg)
\nabla\cdot{\bf \Lambda}_N \, \d x \\
&\leq \frac{1}{2}\int_\Omega 
\eta\bigg|\frac{\partial \Psi_N}{\partial t}\bigg|^2 \d x 
+\frac{1}{2}\int_\Omega 
\eta\kappa^2|\nabla\cdot{\bf \Lambda}_N|^2 \d x .
\end{align*}
By applying Gronwall's inequality, we obtain that
\begin{align*}
&\|\Psi_N\|_{L^\infty((0,T);{\mathcal  H}^1)}
+\|\partial_t\Psi_N\|_{L^2((0,T);{\mathcal  L}^2)}
+\|{\bf\Lambda}_N\|_{L^\infty((0,T);{\bf H}_{\rm n}({\rm curl},{\rm div}))}
+\|\partial_t{\bf\Lambda}_N\|_{L^2((0,T);L^2)}
\leq C,
\end{align*}
where the constant $C$ does not depend on $N$.

Thirdly, since $H^1\hookrightarrow
\hookrightarrow L^p$ for any $1<p<\infty$
and ${\bf H}_{\bf n}({\rm curl},{\rm div})
\hookrightarrow\hookrightarrow 
L^{4+\varepsilon}$ for some $\varepsilon>0$, 
by the Aubin--Lions compactness 
argument \cite{Lions69}, there exist 
\begin{align}
& \psi\in L^\infty((0,T);{\mathcal  H}^1)
\cap H^1((0,T);{\mathcal  L}^2),
\label{regpsiprf}\\
&{\bf A}\in L^\infty((0,T);{\bf H}_{\rm n}({\rm curl},{\rm div}))
\cap H^1((0,T);{\bf L}^2),
\end{align}
and a subsequence of $(\Psi_N,{\bf\Lambda}_N)_{N=1}^\infty$,
denoted by $(\Psi_{N_m},{\bf\Lambda}_{N_m})_{m=1}^\infty$, 
such that 
\begin{align*}
& \Psi_{N_m} \rightharpoonup
\psi\quad\mbox{weakly$^*$ in}~~ L^\infty((0,T);{\mathcal  H}^1) ,\\
& \Psi_{N_m} \rightharpoonup
\psi\quad\mbox{weakly in}~~ L^p((0,T);{\mathcal  H}^1) ~~
\mbox{for any $1<p<\infty$} ,\\
&\partial_t\Psi_{N_m} \rightharpoonup
\partial_t\psi\quad\mbox{weakly in}~~ L^2((0,T);{\mathcal  L}^2) ,\\
& \Psi_{N_m} \rightarrow
\psi\quad \mbox{strongly in}~~ L^p((0,T);{\mathcal  L}^p)~~
\mbox{for any $1<p<\infty$} ,\\
& {\bf\Lambda}_{N_m} \rightharpoonup
{\bf A}\quad\mbox{weakly$^*$ in}~~ 
L^\infty((0,T);{\bf H}_{\rm n}({\rm curl},{\rm div})) ,\\
& {\bf\Lambda}_{N_m} \rightharpoonup
{\bf A}\quad\mbox{weakly in}~~ 
L^p((0,T);{\bf H}_{\rm n}({\rm curl},{\rm div})) ~~
\mbox{for any $1<p<\infty$} ,\\
&\partial_t{\bf\Lambda}_{N_m} \rightharpoonup
\partial_t{\bf A}\quad\mbox{weakly in}~~ L^2((0,T);{\bf L}^2) ,\\
& {\bf\Lambda}_{N_m} \rightarrow
{\bf A}\quad \mbox{strongly in}~~ L^p((0,T);{\bf L}^{4+\varepsilon})
~~\mbox{for any $1<p<\infty$} ,
\end{align*}
which further imply that
\begin{align*}
&\Psi_{N_m}{\bf\Lambda}_{N_m} \rightarrow
\psi{\bf A}\quad\mbox{strongly in}~~ 
L^2((0,T);{\mathcal  L}^2\times {\mathcal  L}^2) ,\\
&\nabla\Psi_{N_m}\cdot{\bf\Lambda}_{N_m} \rightharpoonup
\nabla\psi\cdot{\bf A}\quad
\mbox{weakly in}~~ L^2((0,T);{\mathcal  L}^{4/3}) ,\\
&\Psi_{N_m}|{\bf\Lambda}_{N_m}|^2 \rightarrow
\psi|{\bf A}|^2\quad\mbox{strongly in}~~ 
L^2((0,T);{\mathcal  L}^{4/3}) ,\\
&\bigg(\frac{i}{\kappa} \nabla 
+{\bf \Lambda}_N\bigg) \Psi_N
 \rightharpoonup
\bigg(\frac{i}{\kappa} \nabla 
+{\bf A}\bigg) \psi\quad\mbox{weakly in}~~ 
L^2((0,T);{\mathcal  L}^2) ,\\
&\Psi_N^*\bigg(\frac{i}{\kappa} \nabla 
+{\bf \Lambda}_N\bigg) \Psi_N
 \rightharpoonup
\psi^*\bigg(\frac{i}{\kappa} \nabla 
+{\bf A}\bigg) \psi\quad\mbox{weakly in}~~ 
L^2((0,T);{\mathcal  L}^{4/3}\times {\mathcal  L}^{4/3}) ,\\
&\bigg(\frac{i}{\kappa} \nabla 
+{\bf \Lambda}_N\bigg) \Psi_N\cdot{\bf\Lambda}_N
 \rightharpoonup
\bigg(\frac{i}{\kappa} \nabla 
+{\bf A}\bigg) \psi\cdot{\bf A}
\quad\mbox{weakly in}~~ 
L^2((0,T);{\mathcal  L}^{4/3}) .
\end{align*}
For any given 
$\varphi\in L^2((0,T);{\mathcal V}_N)
\hookrightarrow L^2((0,T);{\mathcal  L}^4)$
and ${\bf a}\in L^2((0,T);{\bf X}_N)
\hookrightarrow L^2((0,T);{\bf L}^4)$,
integrating \refe{DVPDE1}-\refe{DVPDE2}
with respect to time and 
letting $N=N_m\rightarrow \infty$, 
we derive \refe{VPDE1}-\refe{VPDE2}.
In other words,
$\psi\in L^\infty((0,T);{\mathcal H}^1)\cap 
H^1((0,T);{\mathcal L}^2)$ is a weak
solution of \refe{PDE1} in the sense of \refe{VPDE1},
and ${\bf A}\in L^\infty((0,T);{\bf H}_{\rm n}({\rm curl},{\rm div}))
\cap H^1((0,T);{\bf L}^2) $ is a weak solution of \refe{PDE2}
in the sense of \refe{VPDE2}. The conditions of 
Lemma \ref{UnBDPsi} are satisfied, which implies that
$|\psi|\leq 1$ a.e. in $\Omega\times(0,T)$.

Finally, we prove the additional regularity of the solution
specified in Definition \ref{DefWSol}.
From Lemma \ref{LemCDReg} we see that
$${\bf A}\in L^\infty((0,T);{\bf H}_{\rm n}({\rm curl},{\rm div}))
\hookrightarrow L^\infty((0,T);{\bf H}^s)
\hookrightarrow L^\infty((0,T);{\bf L}^4) $$
for any $s\in(1/2,\pi/\omega)$. From \refe{PDE1} we see that 
\begin{align*}
&\frac{1}{\kappa^2}\Delta\psi
=\eta\partial_t\psi
+\frac{i}{\kappa}\nabla\cdot({\bf A}\psi)
+\frac{i}{\kappa}{\bf A}\cdot\nabla\psi+|{\bf A}|^2\psi
+(|\psi|^2-1)\psi-i\eta\kappa\psi\nabla\cdot{\bf A}
 \quad\mbox{a.e. in}~~\Omega,
\end{align*}
which imply that  
\begin{align*}
\|\Delta\psi\|_{{\mathcal  L}^2}
%&\leq \eta\|\partial_t\psi\|_{{\mathcal  L}^2}
%+C\|\psi\nabla\cdot {\bf A}  \|_{{\mathcal  L}^2}
%+C\|{\bf A}\cdot\nabla\psi\|_{{\mathcal  L}^2}
%+\||{\bf A}|^2\psi\|_{{\mathcal  L}^2}
%+\|(|\psi|^2-1)\psi \|_{{\mathcal  L}^2} \\
&\leq C\|\partial_t\psi\|_{{\mathcal  L}^2}
+C\|\nabla\cdot{\bf A}\|_{L^2}\|\psi\|_{{\mathcal  L}^\infty}
+C\|{\bf A}\|_{L^4}\|\nabla\psi\|_{{\mathcal  L}^4}
+C\|{\bf A}\|_{L^4}^2+C\|(|\psi|^2-1)\psi \|_{{\mathcal  L}^2} \\
&\leq C+C\|\partial_t\psi\|_{L^2}
+C\|\nabla\psi\|_{{\mathcal  L}^4}\\
&\leq C+C\|\partial_t\psi\|_{L^2}
+C\|\nabla\psi\|_{{\mathcal  L}^2}^{(1-4/p_s)/(2-4/p_s)}
\|\nabla\psi\|_{{\mathcal  L}^{p_s}}^{1/(2-4/p_s)}\\
&\leq C+\eta\|\partial_t\psi\|_{{\mathcal  L}^2}
+C\|\nabla\psi\|_{{\mathcal  L}^{p_s}}^{1/(2-4/p_s)} \\
&\leq C+C\|\partial_t\psi\|_{{\mathcal  L}^2}
+C\|\Delta\psi\|_{{\mathcal  L}^2}^{1/(2-4/p_s)},
\end{align*}
where we have used \refe{regpsiprf} 
and Lemma \ref{HsPoissEq}  in the last inequality.
Since $1/(2-4/p_s)<1$, 
the last inequality implies   
$\|\Delta\psi\|_{{\mathcal  L}^2}
 \leq C+C\|\partial_t\psi\|_{{\mathcal  L}^2}$ , and so
$$
\|\Delta \psi\|_{L^2((0,T);{\mathcal  L}^2)}\leq   
C+C\|\partial_t\psi\|_{L^2((0,T);{\mathcal  L}^2)}\leq C ,
$$
which further implies $\psi\in L^2((0,T);{\mathcal H}^{1+s})$
by Lemma \ref{HsPoissEq}.
From  \refe{PDE2} we see that
\begin{align*}
&\|\nabla\times(\nabla\times{\bf A})
-\nabla(\nabla\cdot{\bf A})\|_{L^2((0,T);{\bf L}^2)}\\
&\leq C\|\partial_t{\bf A}\|_{L^2((0,T);{\bf L}^2)}
+C\|\psi^*(i\kappa^{-1} \nabla \psi + \mathbf{A} \psi)\|_{L^2((0,T);{\bf L}^2)}
+C\|\nabla\times f\|_{L^2((0,T);{\bf L}^2)}\\
&\leq C\|\partial_t{\bf A}\|_{L^2((0,T);{\bf L}^2)}
+C\| \nabla \psi\|_{L^2((0,T);{\mathcal  L}^2)} 
+C\| \mathbf{A} \|_{L^2((0,T);{\bf L}^2)}
+C\|\nabla\times f\|_{L^2((0,T);{\bf L}^2)}\\
&\leq C .
\end{align*}
Note that $w=\nabla\times{\bf A}-f$ satisfies the equation 
\begin{align*}
&-\Delta w=\nabla\times(\nabla\times w)
= \nabla\times {\bf f} ,
\end{align*}
with $w=0$ on $\partial\Omega$ and
${\bf f}=\nabla\times(\nabla\times{\bf A})
-\nabla(\nabla\cdot{\bf A})-\nabla\times f\in L^2((0,T);{\bf L}^2)$.
The energy estimate of $w$ gives
$$\|w\|_{L^2((0,T);H^1)}\leq C\|{\bf f}\|_{L^2((0,T);{\bf L}^2)}\leq C .$$
Thus $\nabla(\nabla\cdot{\bf A})=\nabla\times w
-{\bf f}\in L^2((0,T);{\bf L}^2)$, which indicates that 
$\nabla\cdot{\bf A}\in L^2((0,T);H^1)$.

Existence of a weak solution of 
\refe{PDE1}-\refe{init} 
in the sense of Definition \ref{DefWSol} has been proved.\medskip

\subsection{Uniqueness of the weak solution}\label{SecUnique}

Suppose that there are two solutions 
$(\psi,{\bf A})$ and $(\Psi,{\bf\Lambda})$ for 
the system \refe{PDE1}-\refe{init} 
in the sense of Definition \ref{DefWSol}.
Let $e=\psi-\Psi$ and ${\bf E}={\bf A}-{\bf\Lambda}$.
Then we have
\begin{align}
&\int_0^T\Big[\big(\eta\partial_t e  ,\varphi\big)
+ \frac{1}{\kappa^2}\big(\nabla  e, \nabla\varphi\big) 
+ \big(|{\bf A}|^2   e,  \varphi\big) \Big]\d t\nn\\
& =\int_0^T\Big[-\frac{i}{\kappa}\big({\bf A}\cdot\nabla e ,\varphi\big)
-\frac{i}{\kappa}\big({\bf E}\cdot\nabla \Psi ,\varphi\big)
+\frac{i}{\kappa}\big(e {\bf A},\nabla\varphi\big)
+\frac{i}{\kappa}\big(\Psi {\bf E},\nabla\varphi\big)  \nn\\
&\quad - \big((|{\bf A}|^2 -|{\bf \Lambda}|^2)  \Psi,  \varphi\big)
 -\big( (|\psi|^{2}-1) \psi-(|\Psi|^{2}-1) \Psi,\varphi\big)\Big]\d t\nn\\
&\quad -\int_0^T\big(i\eta\kappa\psi\nabla\cdot{\bf E}
+i\eta\kappa e\nabla\cdot{\bf \Lambda},\varphi\big)\d t  ,
\label{UErEq1} \\[8pt]
&\int_0^T\Big[\big(\partial_t{\bf E} ,{\bf a}\big)
+ \big(\nabla\times{\bf E},\nabla\times{\bf a}\big)
+\big(\nabla\cdot{\bf E},\nabla\cdot{\bf a}\big) \Big]\d t\nn\\
& =-\int_0^T{\rm Re} \bigg( 
\frac{i}{\kappa}( \psi^*\nabla  \psi- \Psi^*\nabla  \Psi)
+  {\bf A}(|\psi|^2-|\Psi|^2)+|\Psi|^2 {\bf E}\, ,\, {\bf a}\bigg) \d t ,
\label{UErEq2}
\end{align}
for any $\varphi\in L^2((0,T);{\mathcal H}^1)$ and 
${\bf a}\in L^2((0,T);{\bf H}_{\rm n}({\rm curl},{\rm div}))$.
By choosing $\varphi(x,t)=e(x,t)1_{(0,t')}(t)$ 
and ${\bf a}(x,t)={\bf E}(x,t)1_{(0,t')}(t)$, and
using the regularity estimate
$$
{\rm ess}\!\!\sup_{t\in(0,T)}
(\|\nabla \psi\|_{L^2}+\|\nabla \Psi\|_{L^2}
+\|{\bf A}\|_{L^4}+\|{\bf \Lambda}\|_{L^4})\leq C,
$$
we obtain that
\begin{align*}
& \frac{\eta}{2} \|e(\cdot,t') \|_{L^2}^2 
+ \int_0^{t'}\Big(\frac{1}{\kappa^2}\|\nabla  e\|_{L^2}^2
+  \|{\bf A}    e\|_{L^2}^2\Big)\d t \\
&\leq 
\int_0^{t'}\Big(C\|{\bf A}\|_{L^4}\|\nabla e\|_{L^2} \|e\|_{L^4}
+C\|{\bf E}\|_{L^4}\|\nabla \Psi\|_{L^2}\|e\|_{L^4}
+C\| e\|_{L^4} \|{\bf A}\|_{L^4}\|\nabla e\|_{L^2}  \\
&\quad +C\| {\bf E}\|_{L^2}\|\nabla e\|_{L^2}
+C(\|{\bf A}\|_{L^4}+\|{\bf \Lambda}\|_{L^4})
\|{\bf E}\|_{L^2} \|e\|_{L^4} +C\| e\|_{L^2}^2
+C\|\nabla\cdot{\bf E}\|_{L^2}\|e\|_{L^2}\Big)\d t\\
%+C\|\nabla\cdot{\bf \Lambda}\|_{L^2}\|e\|_{L^4}^2 \\
&\leq \int_0^{t'}\Big(C\|\nabla e\|_{L^2}
 (\epsilon^{-1}\|e\|_{L^2}+\epsilon\|\nabla e\|_{L^2})
 +C\|{\bf E}\|_{{\bf H}_{\rm n}({\rm curl},{\rm div})}
 (\epsilon^{-1}\|e\|_{L^2}+\epsilon\|\nabla e\|_{L^2})\\
&\quad
+C\|\nabla e\|_{L^2}(\epsilon^{-1}\|e\|_{L^2}+\epsilon\|\nabla e\|_{L^2}) 
+C\| {\bf E}\|_{L^2}\|\nabla e\|_{L^2}
+C\|{\bf E}\|_{L^2}(\epsilon^{-1}\|e\|_{L^2}+\epsilon\|\nabla e\|_{L^2}) \\
&\quad +C\| e\|_{L^2}^2
+C\|\nabla\cdot {\bf E}\|_{L^2}\| e\|_{L^2} \Big)\d t\\  
%+C(\epsilon^{-1}\|e\|_{L^2}+\epsilon\|\nabla e\|_{L^2})^2 \Big)\d t\\
&\leq \int_0^{t'}\Big(\epsilon\|\nabla e\|_{L^2}^2+
\epsilon\|\nabla\times{\bf E}\|_{L^2}^2
+ \epsilon\|\nabla\cdot{\bf E}\|_{L^2}^2 
+(C+C\epsilon^{-3})\|e\|_{L^2}^2
+ (C+C \epsilon^{-1})\|{\bf E}\|_{L^2}^2\Big)\d t ,
\end{align*}
and
\begin{align*}
& \frac{1}{2}\|{\bf E}(\cdot,t')\|_{L^2}^2 
+\int_0^{t'}\Big(\|\nabla\times{\bf E} \|_{L^2}^2
+\|\nabla\cdot{\bf E} \|_{L^2}^2 \Big)\d t\\
%&\leq \int_0^{t'}\Big(C\| e^*\|_{L^4}
%\|\nabla  \psi\|_{L^{2}}\| {\bf E}\|_{L^4}
%+\|\Psi^*\nabla  e \|_{L^2}\| {\bf E}\|_{L^2}
%+ \|{\bf A}(|\psi|^2-|\Psi|^2)
%+|\Psi|^2 {\bf E}\|_{L^2}\|{\bf E}\|_{L^2}\Big)\d t\\
&\leq \int_0^{t'}\Big(C \| e\|_{L^4}\|\nabla  \psi\|_{L^2}\| {\bf E}\|_{L^4}
+\|\nabla  e \|_{L^2} \| {\bf E}\|_{L^2}
+ (\|e\|_{L^4}\| {\bf A}\|_{L^4}+\|{\bf E}\|_{L^2})\|{\bf E}\|_{L^2}\Big)\d t\\
&\leq \int_0^{t'}\Big(C(\epsilon^{-1}\| e\|_{L^2} 
+\epsilon\|\nabla  e \|_{L^2})
\| {\bf E}\|_{{\bf H}_{\rm n}({\rm curl},{\rm div})}
+\|\nabla  e \|_{L^2} \| {\bf E}\|_{L^2} \\
&\quad + (\|e\|_{L^2} +\|\nabla e\|_{L^2}
+\|{\bf E}\|_{L^2})\|{\bf E}\|_{L^2}\Big)\d t\\
&\leq
\int_0^{t'}\Big(\epsilon\|\nabla e\|_{L^2}^2
+\epsilon\|\nabla\times{\bf E}\|_{L^2}
+\epsilon\|\nabla\cdot{\bf E}\|_{L^2}
+(C+C \epsilon^{-3}) \|e\|_{L^2}^2
+ (C+C \epsilon^{-1})  \|{\bf E}\|_{L^2}^2\Big)\d t ,
\end{align*}
where $\epsilon$ is arbitrary positive number.
By choosing $\epsilon<\frac{1}{4}
\min(1, \kappa^{-2} )$ and summing up the last
two inequalities, we obtain that
\begin{align*}
&  \frac{\eta}{2}\|e(\cdot,t')\|_{L^2}^2
+\frac{1}{2}\|{\bf E}(\cdot,t')\|_{L^2}^2 
\leq 
\int_0^{t'}\Big(C\|e\|_{L^2}^2 
+C\| {\bf E}\|_{L^2}^2\Big)\d t ,
\end{align*}
which implies
\begin{align*}
&\max_{t\in(0,T)}
\bigg(\frac{\eta}{2}\|e\|_{L^2}^2
+\frac{1}{2}\|{\bf E}\|_{L^2}^2\bigg)=0  
\end{align*}
via Gronwall's inequality. 
Uniqueness of the weak solution is proved.\medskip

\subsection{Equivalence of 
\refe{PDE1}-\refe{init} and \refe{RFPDE1}-\refe{RFinit}}

Let $(\psi,{\bf A})$ be the unique solution of 
{\rm\refe{PDE1}-\refe{init}} and, for the given $\psi$ and ${\bf A}$, we let
$(p,q,u,v)$ be the solution of \refe{RFPDEp}-\refe{RFPDEv}.
Since ${\rm Re}\big[\psi^*\big(\frac{i}{\kappa} \nabla 
+ \mathbf{A}\big) \psi\big]\in L^\infty((0,T);L^2)$, 
the standard regularity estimates of 
Poisson's equations yield that
\begin{align*}
&p,q\in L^\infty((0,T);H^1) ,\\
&u,v\in L^\infty((0,T);H^1)\cap L^2((0,T);H^{1+s}),\quad
\partial_tu,\partial_tv,\Delta u,\Delta v\in L^2((0,T);L^2) .
\end{align*}
By setting $\widetilde{\bf A}
=\nabla\times u+\nabla v$, we have
$\widetilde{\bf A}\in L^\infty((0,T);{\bf L}^2)
\cap L^2((0,T);{\bf H}_{\rm n}({\rm curl},{\rm div}))$
and $\partial_t\widetilde {\bf A}\in L^2((0,T);({\bf H}_{\rm n}({\rm curl},{\rm div}))')$. 
Since ${\rm Re}\big[\psi^*\big(\frac{i}{\kappa} \nabla 
+ \mathbf{A}\big) \psi\big]=\nabla\times p+\nabla q$, 
the integration of \refe{RFPDEu} against $\nabla\times{\bf a}$
minus the integration of \refe{RFPDEv} 
against $\nabla\cdot{\bf a}$ gives
\begin{align*}
&\int_0^T\bigg[\bigg(\frac{\partial \widetilde{\bf A}}{\partial t},{\bf a}\bigg) 
+ \big(\nabla\times\widetilde{\bf A},\nabla\times{\bf a}\big)
+\big(\nabla\cdot\widetilde{\bf A},\nabla\cdot{\bf a}\big)  \bigg]\d t \nn\\
&=\int_0^T  \big(f,\nabla\times{\bf a}\big) \d t
-\int_0^T \bigg({\rm Re}\bigg[\psi^*\bigg(\frac{i}{\kappa} \nabla 
+ \mathbf{A}\bigg) \psi\bigg],{\bf a}\bigg) \d t
\end{align*}
for any ${\bf a}\in L^2((0,T);{\bf H}_{\rm n}({\rm curl},{\rm div}))$, 
with $\widetilde{\bf A}_0={\bf A}_0$. Comparing the above equation
with \refe{VPDE2}, we derive that $\widetilde{\bf A}= {\bf A}$.
Thus $\Delta u=\nabla\times{\bf A}\in L^\infty((0,T);L^2)\cap L^2((0,T);H^1)$
and $\Delta v=\nabla\cdot{\bf A}\in L^\infty((0,T);L^2)\cap L^2((0,T);H^1)$,
and from \refe{RFPDEu}-\refe{RFPDEv} we further derive that
$\partial_tu,\partial_tv\in L^\infty((0,T);L^2)\cap L^2((0,T);H^1)$.

Overall, \refe{RFPDE1}-\refe{RFinit} has a solution
$(\psi,p,q,u,v)$ which possesses the regularity 
specified in Definition \ref{DefWSol2},
satisfying \refe{VPDE1-2}-\refe{RFPDEv-2} 
with ${\bf A}=\nabla\times u+\nabla v$,
where $(\psi,{\bf A})$ coincides with the unique solution of
{\rm\refe{PDE1}-\refe{init}}.
Based on the regularity of 
$\psi,p,q,u$ and $v$, 
uniqueness of the solution for \refe{RFPDE1}-\refe{RFinit} 
can be proved in a similar way as Section \ref{SecUnique}.
We omit the proof due to the limitation on pages.\medskip

\subsection{Singularity of the solution}
 
From the analysis in the last two subsections we see that
$$
\left\{\begin{array}{ll}
-\Delta u=\nabla\times {\bf A}  &\mbox{in}~~\Omega,\\
u=0  &\mbox{on}~~\partial\Omega ,
\end{array}\right.
\qquad\mbox{and} \qquad
\left\{\begin{array}{ll}
\Delta v=\nabla\cdot{\bf A}  &\mbox{in}~~\Omega,\\
\partial_nv=0  &\mbox{on}~~\partial\Omega ,
\end{array}\right.
$$
where $\nabla\times {\bf A},
\nabla\cdot{\bf A}\in L^\infty((0,T);L^2)$.
For each fixed $t$, the solutions of 
the two Poisson's equations 
have the decomposition \cite{Kellogg}
\begin{align*}
&u(x,t)=\sum_{j=1}^m\beta_j(t)\Phi(|x-x_j|)|x-x_j|^{\pi/\omega_j}
\sin(\pi\Theta_j(x)/\omega_j) +\widetilde u(x,t),\\
&v(x,t)=\sum_{j=1}^m\gamma_j(t)\Phi(|x-x_j|)|x-x_j|^{\pi/\omega_j}
\cos(\pi\Theta_j(x)/\omega_j) +\widetilde v (x,t),
\end{align*}
where 
\begin{align*}
&\sum_{j=1}^m|\beta_j(t)|+\sum_{j=1}^m|\gamma_j(t)|
+\|\widetilde u (\cdot,t)\|_{H^2}+\|\widetilde v( \cdot,t)\|_{H^2}
\leq C\|\nabla\times {\bf A}\|_{L^2}+C\|\nabla\cdot{\bf A}\|_{L^2}.
\end{align*}
Thus 
\begin{align*}
&\|\beta_j\|_{L^\infty(0,T)}+\|\gamma_j\|_{L^\infty(0,T)}
+\|\widetilde u \|_{L^\infty((0,T);H^2)}+\|\widetilde v \|_{L^\infty((0,T);H^2)}\\
&\leq C(\|\nabla\times {\bf A}\|_{L^\infty((0,T);L^2)}
+ \|\nabla\cdot{\bf A}\|_{L^\infty((0,T);L^2)}) 
\leq C.
\end{align*}
The singular part of $\psi$ can be derived 
in a similar way.

The proof of Theorem \ref{MainTHM1} is completed. 
\bigskip

\section{Proof of Theorem \ref{MainTHM2}}
 
In this section, we prove further regularity of the solution 
under some compatibility conditions. 
%For simplicity,
%we only present a priori estimates of the solution.
We need the following lemma concerning the maximal
$L^p$ regularity of parabolic equations 
in a Lipscthiz domain \cite{Wood07}.
\begin{lemma}\label{MaxLp}
{\it The solution of the equation
\begin{align*}
\left\{\begin{array}{ll}
\partial_tu-\Delta u=f  &\mbox{in}~~\Omega,\\
u=0  &\mbox{on}~~\partial\Omega ,\\
u(x,0)=0 &\mbox{for}~~x\in\Omega 
\end{array}\right.
\qquad\mbox{and}\qquad
\left\{\begin{array}{ll}
\partial_tv-\Delta v=g  &\mbox{in}~~\Omega,\\
\partial_nv=0  &\mbox{on}~~\partial\Omega ,\\
v(x,0)=0 &\mbox{for}~~x\in\Omega 
\end{array}\right.
\end{align*}
satisfy that
\begin{align*}
&\|\partial_t u\|_{L^p((0,T);L^2)}+\|\Delta u\|_{L^p((0,T);L^2)}\leq
C_p\|f\|_{L^p((0,T);L^2)} ,\\
&\|\partial_t v\|_{L^p((0,T);L^2)}+\|\Delta v\|_{L^p((0,T);L^2)}\leq
C_p\|g\|_{L^p((0,T);L^2)} ,
\end{align*}
for any $1<p<\infty$.
}
\end{lemma}

%From Theorem \ref{MainTHM1} we see that 
%$\psi\in L^2((0,T);{\mathcal  H}^2(D))$ and
%${\bf A}\in L^2((0,T);{\bf H}^2(D))$
%for any subdomain $D\subset\subset\Omega$.
%Thus the equations \refe{PDE1}-\refe{PDE2} hold 
%pointwise a.e. in $\Omega\times(0,T)$.

Rewriting \refe{PDE1} as
\begin{align*}
&\eta\frac{\partial (\psi-\psi_0)}{\partial t} 
-\frac{1}{\kappa^2}\Delta (\psi-\psi_0)
 = -g, 
\end{align*}
with 
$$
g=
\frac{i}{\kappa}\nabla\cdot({\bf A}\psi)
+\frac{i}{\kappa}{\bf A}\cdot\nabla\psi
+|{\bf A}|^2\psi+(|\psi|^2-1)\psi
-i\eta\kappa\psi\nabla\cdot{\bf A}-\frac{i}{\kappa}\Delta\psi_0 ,
$$
and applying Lemma \ref{MaxLp} 
(here we need the compatibility condition 
$\partial_n\psi_0=0$ on $\partial\Omega$), 
we derive that, for any given $1<p<\infty$,
\begin{align*}
&\|\partial_t(\psi-\psi_0) \|_{L^p((0,T);L^2)}
 +\|\Delta(\psi-\psi_0)\|_{L^p((0,T);L^2)}\\
 &\leq C\|g\|_{L^p((0,T);L^2)} \\
&\leq C\|\nabla\cdot{\bf A}\|_{L^p((0,T);L^2)}
+C\|{\bf A}\|_{L^{\infty}((0,T);L^4)}\|\nabla\psi\|_{L^p((0,T);L^4)} 
 +C\|\mathbf{A} \|_{L^{2p}((0,T);L^4)}^2
 +C \\
%&\leq C\|\nabla\cdot{\bf A}\|_{L^p((0,T);L^2)}
%+C\|\nabla\psi\|_{L^p((0,T);L^4)}  +C\\
&\leq C 
+C\|\nabla\psi\|_{L^p((0,T);L^2)}^{(1-4/p_s)/(2-4/p_s)}
\|\nabla\psi\|_{L^p((0,T);L^{p_s})}^{1/(2-4/p_s)}  +C\\
&\leq  
C\|\nabla\psi\|_{L^p((0,T);L^{p_s})}^{1/(2-4/p_s)}  +C\\
&\leq  
C\|\Delta\psi\|_{L^p((0,T);L^{2})}^{1/(2-4/p_s)}  +C,
\end{align*}
which implies that
$
\|\partial_t\psi \|_{L^p((0,T);L^2)}
+\|\Delta\psi\|_{L^p((0,T);L^2)}
\leq C.
$
In other words, we have
\begin{align}
\psi\in \bigcap_{p>1}W^{1,p}((0,T);L^2)\cap L^p((0,T);H^{1+s})
\hookrightarrow L^\infty((0,T);W^{1,4}) .
\end{align}

Let $\overline w=\nabla\cdot{\bf A}$ and consider 
the divergence of \refe{PDE2}, i.e.
\begin{align*}
&\frac{\partial \overline w}{\partial t}  
-\Delta \overline w    
=  - {\rm Re}\bigg[\nabla\psi^*\cdot\bigg(\frac{i}{\kappa} \nabla\psi 
+ \mathbf{A}\psi\bigg) 
 + \psi^*\bigg(\frac{i}{\kappa} \Delta\psi+ \psi\nabla\cdot\mathbf{A}  
 +\mathbf{A}\cdot\nabla\psi\bigg)\bigg] ,
\end{align*}
with the boundary condition $\partial_nw=0$ on $\partial\Omega$.
The standard energy estimates of the above equation give 
\begin{align*}
&\|\partial_t\overline w\|_{L^2((0,T);L^2)} 
+\|\Delta\overline  w \|_{L^2((0,T);L^2)}\\
&\leq C\|\overline w_0\|_{H^1}
+C\big\|\nabla\psi^*\cdot\big(i\kappa^{-1} \nabla\psi 
+ \mathbf{A}\psi\big) 
 + \psi^*\big(i\kappa^{-1} \Delta\psi
 + \psi\nabla\cdot\mathbf{A}  
 +\mathbf{A}\cdot\nabla\psi\big)\big\|_{L^2((0,T);L^2)}\\
&\leq C\|\nabla\cdot{\bf A}_0\|_{H^1}
+C\|\nabla\psi^*\|_{L^4((0,T);L^4)}(\| \nabla\psi\|_{L^4((0,T);L^4)} 
+ \|\mathbf{A}\|_{L^4((0,T);L^4)}) \\
&~~~ + C(\| \Delta\psi\|_{L^2((0,T);L^2)}+\|  \nabla\cdot\mathbf{A}\|_{L^2((0,T);L^2)}  
+\|\mathbf{A}\|_{L^4((0,T);L^4)}\|\nabla\psi\|_{L^4((0,T);L^4)})\\
&\leq C
\end{align*}
If we let $w=\nabla\times{\bf A}-f$ and 
consider the curl of \refe{PDE2},
in a similar way one can prove 
\begin{align*}
&\|\partial_tw\|_{L^2((0,T);L^2)} 
+\|\Delta w \|_{L^2((0,T);L^2)}\leq C .
\end{align*}
The last two inequalities imply that
\begin{align}
\partial_t{\bf A}\in L^2((0,T);{\bf H}_{\rm n}({\rm curl},{\rm div}))
\hookrightarrow L^2((0,T);{\bf L}^4) .
\end{align}

Consider the time derivative of \refe{PDE1} 
and denote $\dot\psi=\partial_t\psi$.
We have 
\begin{align*}
&\eta\frac{\partial \dot\psi}{\partial t} 
 -\frac{1}{\kappa^2}\Delta\dot\psi=-  \dot g,
\end{align*}
with the boundary condition 
$\partial_n\dot\psi=0$ on $\partial\Omega$,
where
\begin{align*}
\dot g&=\big(i\kappa^{-1}-i\eta\kappa\big)\dot\psi \nabla\cdot{\bf A}
 +\big(i\kappa^{-1}-i\eta\kappa\big)\psi \nabla\cdot\dot{\bf A}
+2i\kappa^{-1}\dot{\bf A}\cdot\nabla\psi 
+2i\kappa^{-1}{\bf A}\cdot\nabla\dot\psi \\
&~~~ + 2{\bf A}\cdot\dot{\bf A}  \psi
+|{\bf A}|^2\dot\psi
 + (\dot\psi\psi^*+\psi\dot\psi^*) \psi 
  + (|\psi|^2-1)\dot\psi .
\end{align*}
The energy estimates of the equation give that
\begin{align*}
&\|\partial_{t}\dot\psi\|_{L^2((0,T);L^2)}
+\|\Delta\dot\psi\|_{L^2((0,T);L^2)}
+\|\nabla\dot\psi\|_{L^\infty((0,T);L^2)}\\
&\leq C\|g\|_{L^2((0,T);L^2)}\\
&\leq C\|\dot\psi\|_{L^\infty((0,T);L^2)}
\|\nabla\cdot{\bf A}\|_{L^2((0,T);L^\infty)}
+C\|\nabla\cdot\dot{\bf A}\|_{L^2((0,T);L^2)} \\
&\quad +C\|\dot{\bf A}\|_{L^2((0,T);L^4)}
\|\nabla\psi \|_{L^\infty((0,T);L^4)}
+C\|{\bf A}\|_{L^\infty((0,T);L^4)}
\|\nabla\dot\psi \|_{L^2((0,T);L^4)} \\
&\quad 
+C\|{\bf A}\|_{L^\infty((0,T);L^4)}\|\dot{\bf A}\|_{L^2((0,T);L^4)}
+C\|{\bf A}\|_{L^\infty((0,T);L^4)}^2\|\dot\psi\|_{L^2((0,T);L^\infty)}
+C\|\dot\psi\|_{L^2((0,T);L^2)}\\
&\leq C\|\dot\psi\|_{L^2((0,T);L^2)}^{1/2}
\|\partial_{t}\dot\psi\|_{L^2((0,T);L^2)}^{1/2}
+C+C+C\|\nabla\dot\psi \|_{L^2((0,T);L^4)} 
+C+C\|\nabla\dot\psi \|_{L^2((0,T);L^4)}   \\
&\leq C\|\partial_{t}\dot\psi\|_{L^2((0,T);L^2)}^{1/2}
+C\|\nabla\dot\psi\|_{L^2((0,T);L^2)}^{\frac{1-4/p_s}{2-4/p_s}}
\|\nabla\dot\psi\|_{L^2((0,T);L^{p_s})}^{\frac{1}{2-4/p_s}} +C\\
&\leq C\|\partial_{t}\dot\psi\|_{L^2((0,T);L^2)}^{1/2}
+C\|\dot\psi\|_{L^2((0,T);L^2)}^{\frac{(1-4/p_s)s}{(2-4/p_s)(1+s)}}
\|\dot\psi\|_{L^2((0,T);H^{1+s})}^{\frac{1-4/p_s}{(2-4/p_s)(1+s)}}
\|\nabla\dot\psi\|_{L^2((0,T);L^{p_s})}^{\frac{1}{2-4/p_s}} +C\\
%&\leq C\|\partial_{t}\dot\psi\|_{L^2((0,T);L^2)}^{1/2}
%+C\|\Delta\dot\psi\|_{L^2((0,T);L^2)}^{\frac{1-4/p_s}{(2-4/p_s)(1+s)}
%+\frac{1}{2-4/p_s}}  +C\\
&\leq C\|\partial_{t}\dot\psi\|_{L^2((0,T);L^2)}^{1/2}
+C\|\Delta\dot\psi\|_{L^2((0,T);L^2)}^{1-\frac{(1-4/p_s)s}{(2-4/p_s)(1+s)}
}  +C,
\end{align*}
which reduces to
%\|\nabla\dot\psi\|_{L^2((0,T);L^2)}^{(1-4/p_s)/(2-4/p_s)}
%\|\nabla\dot\psi\|_{L^2((0,T);L^{p_s})}^{1/(2-4/p_s)} 
%
\begin{align*}
&\|\partial_{t}\dot\psi\|_{L^2((0,T);L^2)}+\|\Delta\dot\psi\|_{L^2((0,T);L^2)}
+\|\nabla\dot\psi\|_{L^\infty((0,T);L^2)}\leq C .
\end{align*}
In other words, we have
\begin{align}
&\|\partial_{tt}\psi\|_{L^2((0,T);L^2)}+\|\partial_t\psi\|_{L^2((0,T);H^{1+s})}
+\|\partial_t\psi\|_{L^\infty((0,T);H^1)}\leq C.
\end{align}

Now we consider the time derivative of \refe{RFPDEp}-\refe{RFPDEv}, i.e.
\begin{align}
&\Delta \dot p =-\nabla\times  {\rm Re}\big[\psi^*
\big(i\kappa^{-1} \nabla + \mathbf{A}\big)
 \psi\big]^{\mbox{\large$\cdot$}}
\label{RFPDEpd}\\[5pt]
&\Delta \dot q =\nabla\cdot {\rm Re}\big[\psi^*
\big(i\kappa^{-1}\nabla + \mathbf{A}\big) 
\psi\big]^{\mbox{\large$\cdot$}} 
 \label{RFPDEqd}\\[5pt]
&\frac{\partial \dot u}{\partial t} -\Delta \dot u
=  \dot f-\dot p ,\label{RFPDEud}\\[5pt]
&\frac{\partial \dot v}{\partial t} -\Delta \dot v
=  -\dot q ,
\label{RFPDEvd}
\end{align}
with the boundary conditions $\dot p=0$, 
$\partial_n\dot q=0$, $\dot u=0$ and 
$\partial_n\dot v=0$ on $\partial\Omega$. In particular,
the boundary condition $\dot u=0$ on $\partial\Omega$
at the time $t=0$
requires the compatibility condition $\nabla\times{\bf A}_0=f_0$ 
on $\partial\Omega$. 
Since 
\begin{align*}
&\big\|\big[\psi^*
\big(i\kappa^{-1}\nabla + \mathbf{A}\big) 
\psi\big]^{\mbox{\large$\cdot$}}\,\big\|_{L^2((0,T);L^2)}\\
&=\big\|\dot\psi^*\big(i\kappa^{-1} \nabla 
+ \mathbf{A}\big) \psi
+\psi^*\big(i\kappa^{-1} \nabla 
+ \mathbf{A}\big) \dot\psi
+|\psi|^2\dot{\bf A}\big\|_{L^2((0,T);L^2)}\\
&\leq \|\dot\psi^*\|_{L^2((0,T);L^\infty)}
\|i\kappa^{-1} \nabla \psi + \mathbf{A} \psi\|_{L^\infty((0,T);L^2)}\\
&\quad 
+\| i\kappa^{-1} \nabla \dot\psi\|_{L^2((0,T);L^2)}
+ \|\mathbf{A}\|_{L^\infty((0,T);L^2)}\| \dot\psi\|_{L^2((0,T);L^\infty)}
+\| \dot{\bf A}\|_{L^2((0,T);L^2)}\\
&\leq C ,
\end{align*}
the energy estimates of \refe{RFPDEpd}-\refe{RFPDEqd} give
$$
\|\nabla \dot p\|_{L^2((0,T);L^2)}+\|\nabla \dot q\|_{L^2((0,T);L^2)} 
\leq \big\|\big[\psi^*
\big(i\kappa^{-1} \nabla + \mathbf{A}\big) 
\psi\big]^{\mbox{\large$\cdot$}}\,\big\|_{L^2((0,T);L^2)}
\leq C,
$$
and then the energy estimates of \refe{RFPDEud}-\refe{RFPDEvd} give
\begin{align*}
&\|\partial_t\dot u\|_{L^2((0,T);L^2)}+\|\Delta\dot u\|_{L^2((0,T);L^2)}
+\|\nabla\dot u\|_{L^\infty((0,T);L^2)}\leq C\|\dot f-\dot p\|_{L^2((0,T);L^2)}
\leq C ,\\
&\|\partial_t\dot v\|_{L^2((0,T);L^2)}+\|\Delta\dot v\|_{L^2((0,T);L^2)}
+\|\nabla\dot v\|_{L^\infty((0,T);L^2)}\leq C\|\dot q\|_{L^2((0,T);L^2)} 
\leq C,
\end{align*}
which further imply that
$\partial_tu, \partial_tv\in L^2((0,T);H^{1+s}) $.
%\begin{align*}
%&\|\partial_{tt}u\|_{L^2((0,T);L^2)}
%+\|\Delta\partial_tu\|_{L^2((0,T);L^2)}
%+\|\partial_tu\|_{L^2((0,T);H^{1+s})}
%+\|\partial_tu\|_{L^\infty((0,T);H^1)}\leq C,\\
%&\|\partial_{tt}v\|_{L^2((0,T);L^2)}+\|\Delta\partial_tv\|_{L^2((0,T);L^2)}
%+\|\partial_tv\|_{L^2((0,T);H^{1+s})}
%+\|\partial_tv\|_{L^\infty((0,T);H^1)}\leq C .
%\end{align*}

The proof of Theorem \ref{MainTHM2} is completed. \bigskip

\section{Proof of Theorem \ref{MainTHM3}}
\setcounter{equation}{0}

The proof consists of two parts.
In the first part, we prove the boundedness of the finite
element solution and the invertibility of the linear systems,
which are independent of the regularity of the exact solution.
In the second part, we present error estimates of the 
finite element solution based on a mathematical
induction on the $L^4$ norm of 
${\bf A}_h^n=\nabla\times u_h^n+\nabla v_h^n$,
which is needed to control the nonlinear terms in the equations.
%Note that the $L^4$ norm estimate of ${\bf A}_h^n$ is almost critical
%in the sense that for any given $p_0>4$  there exists a 
%nonconvex polygon such that
%the exact solution ${\bf A}^n$ does not belong to ${\bf L}^{p_0}$.

\subsection{Stability of the finite element solution}

Substituting $\varphi=\psi_h^{n+1}$ into \refe{FEMEq1}
and considering the real part, we derive that
\begin{align*}
& D_\tau\bigg(\frac{\eta}{2} \|\psi^{n+1}_h\|_{L^2}^2\bigg) + 
\|(i\kappa^{-1}\nabla + \mathbf{A}^{n}_h) \psi^{n+1}_h \|_{L^2}^2 +
\int_\Omega |\psi^{n}_h|^{2} |\psi^{n+1}_h|^2\d x
= \|\psi^{n+1}_h\|_{L^2}^2 ,
\end{align*}
which together with the discrete Gronwall's inequality 
implies that, when $\tau<\eta/4$,
\begin{align}\label{APEPsi}
&\max_{0\leq n\leq N-1} \|\psi^{n+1}_h\|_{L^2}^2  + 
\sum_{n=0}^{N-1}\tau \|(i\kappa^{-1}\nabla + \mathbf{A}^{n}_h) \psi^{n+1}_h \|_{L^2}^2  \leq C .
\end{align}
Since $|\chi(\psi_h^{n})|\leq 1$, 
by substituting $\xi=p_h^{n+1}$ into \refe{FEMEq2}
and substituting $\zeta=q_h^{n+1}$ into \refe{FEMEq3}, 
we obtain
\begin{align*}
&\|\nabla p^{n+1}_h\|_{L^2} +\|\nabla q^{n+1}_h\|_{L^2} 
\leq C\|(i\kappa^{-1}\nabla 
+ \mathbf{A}^n_h) \psi^{n+1}_h\|_{L^2} ,
\end{align*}
which together with \refe{APEPsi} gives
\begin{align}
&\sum_{n=0}^{N-1}\tau\|\nabla p^{n+1}_h\|_{L^2} 
+\sum_{n=0}^{N-1}\tau\|\nabla q^{n+1}_h\|_{L^2} 
\leq C .
\end{align}
Then, substituting $\theta=D_\tau u_h^{n+1}$ into \refe{FEMEq4}
and $\vartheta=D_\tau v_h^{n+1}$ into \refe{FEMEq5},
we derive that
\begin{align}\label{APEpq}
& \sum_{n=0}^{N-1}\tau 
\big(\|D_\tau u_h^{n+1}\|_{L^2}^2 +\|D_\tau v_h^{n+1}\|_{L^2}^2 \big)
+\max_{0\leq n\leq N-1}  \big( \|\nabla u_h^{n+1}\|_{L^2}^2  
+ \|\nabla v_h^{n+1}\|_{L^2}^2\big) \nn\\
&\leq C\sum_{n=0}^{N-1}\tau\|f^{n+1}\|_{L^2} 
+C\sum_{n=0}^{N-1}\tau\|p^{n+1}_h\|_{L^2} 
+C\sum_{n=0}^{N-1}\tau\|q^{n+1}_h\|_{L^2}    \leq C  .
\end{align}
%From the last inequality and the equations 
%\refe{FEMEq4}-\refe{FEMEq5}
%we also derive that
%\begin{align}\label{APEpq}
%& \sum_{n=0}^{N-1}\tau 
%\big(\|\Delta_h u_h^{n+1}\|_{L^2}^2 
%+\|\Delta_h v_h^{n+1}\|_{L^2}^2 \big)
%\leq C  .
%\end{align}

From the above derivations it is not difficult to see that
the linear systems defined by 
\refe{FEMEq1}-\refe{FEMEq5} are invertible when $\tau<\eta/4$,
and the discrete solution 
$(\psi_h^{n},p_h^{n},q_h^{n},
u_h^{n},v_h^{n})$ solved from \refe{FEMEq1}-\refe{FEMEq5}
is uniformly bounded in 
$L^\infty_\tau({\mathcal L}^2)\times L^2_\tau(H^1)
\times L^2_\tau(H^1)\times 
L^\infty_\tau(H^1)\times L^\infty_\tau(H^1)$
with respect to the time-step size $\tau$ and spatial mesh size $h$. 
%We see that the invertibility of the linear system
%and the stability of the finite element solution do not depend
%on the regularity of the exact solution.

\subsection{Error estimates}

Note that the exact solution $(\psi,p,q,u,v)$ satisfies the equations 
\begin{align}
&\big(D_\tau \psi^{n+1}, \varphi\big) +\big
((i\kappa^{-1}\nabla + \mathbf{A}^{n})
\psi^{n+1},(i\kappa^{-1}\nabla + \mathbf{
A}^{n})  \varphi\big) \nn\\
&\qquad\, +\big
((|\psi^{n}|^{2}-1) \psi^{n+1},
\varphi\big) 
+\big(i\eta\kappa {\bf A}^n,\nabla((\psi^{n+1})^* \varphi\big)
= \big(E_\psi^{n+1}, \varphi\big),&
\label{ExVEq1}\\
&(\nabla p^{n+1},\nabla \xi)=\big({\rm Re}
[\chi(\psi^n)^*(i\kappa^{-1}\nabla\psi^{n+1}
+ \mathbf{A}^n \psi^{n+1})],
\nabla\times \xi\big)
+ \big(E_p^{n+1},\nabla\times \xi\big)
\label{ExVEq2}\\
&(\nabla q^{n+1},\nabla \zeta)
=\big({\rm Re}
[\chi(\psi^n)^*(i\kappa^{-1}\nabla\psi^{n+1} 
+ \mathbf{A}^n \psi^{n+1})],\nabla \zeta\big)
+ \big(E_q^{n+1},\nabla \zeta\big)
\label{ExVEq3}\\
&\big(D_\tau u^{n+1}, \theta\big) 
+\big(\nabla u^{n+1},\nabla\theta\big)
=(f^{n+1}-p^{n+1},\theta)
+ \big(E_u^{n+1}, \theta\big) 
\label{ExVEq4}\\
&\big(D_\tau v^{n+1}, \vartheta\big) 
+\big(\nabla v^{n+1},\nabla\vartheta\big)
=(-q^{n+1},\vartheta)
+ \big(E_v^{n+1}, \vartheta\big)
,\label{ExVEq5}
\end{align}
for all $\varphi\in{\mathcal  V}^1_h$, $\xi,\theta\in \mathring V^1_h$
and $\zeta,\vartheta\in V^1_h$, 
with 
\begin{align}
&(\nabla u^{0},\nabla \xi)=\big({\bf A}_0,
\nabla\times \xi\big)
,\quad\forall~\xi\in\mathring V_h^1 ,
\label{ExVEqu0}\\
&(\nabla v^{0},\nabla \zeta)
=\big({\bf A}_0,\nabla\cdot \zeta\big)
,\quad\,\,\, \forall~\zeta\in V_h^1 ,
\label{ExVEqv0}
\end{align}
where
\begin{align*}
E_\psi^{n+1}
&=\eta( D_\tau \psi^{n+1}- \partial_t \psi^{n+1} )
+\frac{i}{\kappa}\nabla\cdot(({\bf A}^n-{\bf A}^{n+1})\psi^{n+1})
+\frac{i}{\kappa}({\bf A}^n-{\bf A}^{n+1})\cdot\nabla\psi^{n+1}\\
&~~~+(|{\bf A}^n|^2-|{\bf A}^{n+1}|^2)\psi^{n+1}
+(|\psi^{n}|^2-|\psi^{n+1}|^2)\psi^{n+1}
+i\eta\kappa\psi^{n+1}\nabla\cdot({\bf A}^{n+1}-{\bf A}^n)\\[5pt]
%%%%%%%%%%%%%%%%%%%%%%%%%%%%%%
E_p^{n+1}
&=E_q^{n+1}={\rm Re}
[i\kappa^{-1}(\psi^{n+1}-\psi^n)^* \nabla\psi^{n+1}
+((\psi^{n+1})^* {\bf A}^{n+1}-(\psi^{n})^* {\bf A}^{n}) \psi^{n+1})\\[5pt]
%%%%%%%%%%%%%%%%%%%%%%%%%%%%%%
E_u^{n+1}
&=D_\tau u^{n+1}-\partial_tu^{n+1}\\[5pt]
%%%%%%%%%%%%%%%%%%%%%%%%%%%%%%
E_v^{n+1}
&=D_\tau v^{n+1}-\partial_tv^{n+1} ,
\end{align*}
are truncation errors due to the time
discretization, which satisfy that
\begin{align*}
\sum_{n=0}^{N-1}\tau
\big(
\|E_\psi^{n+1}\|_{L^2}^2
+\|E_p^{n+1}\|_{L^2}^2
+\|E_q^{n+1}\|_{L^2}^2
+\|E_u^{n+1}\|_{L^2}^2
+\|E_v^{n+1}\|_{L^2}^2
\big) \leq C\tau^2 .
\end{align*}

Let $R_h:{\mathcal  H}^1\rightarrow {\mathcal  V}_h^1$ 
and $\mathring R_h :\mathring H^1\rightarrow \mathring V_h^1$
denote the Ritz projection operator onto the
finite element spaces, i.e.
\begin{align*}
&(\nabla(\phi-R_h\phi),\nabla \varphi)=0
\quad\forall~\phi\in {\mathcal  H}^1~~
\mbox{and}~~\varphi\in {\mathcal  V}_h^1,\\
&(\nabla(\phi-\mathring R_h\phi),\nabla \varphi)=0
\quad\forall~\phi\in \mathring H^1~~\mbox{and}~~
\varphi\in \mathring V_h^1 .
\end{align*}
Then $R_h$, restricted to $H^1$, is just the Ritz projection from
$H^1$ onto $V_h^1$, and we have \cite{BS02,CLTW}
\begin{align*}
&\|\phi-R_h\phi\|_{{\mathcal  L}^2}
+h^s\|\nabla(\phi-R_h\phi)\|_{{\mathcal  L}^2}
\leq Ch^{2s}\|\phi\|_{{\mathcal  H}^{1+s}},
\qquad\, \forall~\phi \in {\mathcal  H}^{1+s},\\
%&\|\nabla(\phi-R_h\phi)\|_{{\mathcal  L}^2}
%\leq Ch^{s}\|\phi\|_{{\mathcal  H}^{1+s}} ,
%\quad\forall~\phi \in {\mathcal  H}^{1+s},\\
%&\|\phi-R_h\phi\|_{L^2}+h\|\nabla(\phi-R_h\phi)\|_{L^2}
%\leq Ch^{r+1}\|\phi\|_{H^{r+1}} ,\quad\forall~\phi \in H^{r+1},\\
&\|\phi-\mathring R_h\phi\|_{L^2}
+h^s\|\nabla(\phi-\mathring R_h\phi)\|_{L^2}
\leq Ch^{2s}\|\phi\|_{H^{1+s}},
\qquad\, \forall~\phi \in \mathring H^1\cap H^{1+s}.
%&\|\nabla(\phi-\mathring R_h\phi)\|_{L^2}
%\leq Ch^{s}\|\phi\|_{H^{1+s}},
%\quad\forall~\phi \in \mathring H^1\cap H^{1+s}  .
\end{align*}

Let $e_{\psi,h}^{n+1}=\psi_h^{n+1}-R_h\psi^{n+1}$,
$e_{p,h}^{n+1}=p_h^{n+1}-\mathring R_hp^{n+1}$,
$e_{q,h}^{n+1}=q_h^{n+1}-R_hq^{n+1}$,
$e_{u,h}^{n+1}=u_h^{n+1}-\mathring R_hu^{n+1}$,
$e_{v,h}^{n+1}=v_h^{n+1}-R_hv^{n+1}$.
The difference between \refe{FEMEq1}-\refe{FEMEqv0} 
and \refe{ExVEq1}-\refe{ExVEqv0} gives that
$u_h^0=\mathring R_hu^0$, 
$v_h^0= R_hv^0$ and
\begin{align}
&\big(\eta D_\tau e_{\psi,h}^{n+1}, \varphi\big) + 
\kappa^{-2}\big(\nabla  
e_{\psi,h}^{n+1}, \nabla   \varphi\big) \nn \\
&=\big(\eta D_\tau (\psi^{n+1}-R_h\psi^{n+1}), 
\varphi\big)  -\big(E_\psi^{n+1}, \varphi\big)
-\frac{i}{\kappa}\big({\bf A}^{n}_h\cdot
\nabla e_{\psi,h}^{n+1} ,\varphi\big)
  \nn \\
&\quad +\frac{i}{\kappa}\big({\bf A}^{n}_h
\cdot\nabla (\psi^{n+1}-R_h\psi^{n+1}) ,\varphi\big)
 -\frac{i}{\kappa}\big(({\bf A}_h^{n}-{\bf A}^{n})
\cdot\nabla \psi^{n+1} ,\varphi\big)
\nn\\
&\quad 
+\frac{i}{\kappa}\big(e_{\psi,h}^{n+1}{\bf A}_h^{n},\nabla\varphi\big)
-\frac{i}{\kappa}\big((\psi^{n+1}
-R_h\psi^{n+1}){\bf A}_h^{n},\nabla\varphi\big) 
+\frac{i}{\kappa}\big(\psi^{n+1}
({\bf A}_h^{n}-{\bf A}^{n}),\nabla\varphi\big)  
 \nn\\
&\quad 
 - \big((|{\bf A}_h^n|^2 -|{\bf A}^n|^2)  \psi^{n+1},  \varphi\big)
 - \big( |{\bf A}^n_h|^2e_{\psi,h}^{n+1},  \varphi\big) 
 - \big( |{\bf A}^n_h|^2(\psi^{n+1}-R_h\psi^{n+1}),  \varphi\big) 
  \nn\\
&\quad
-\big( (|\psi_h^{n}|^{2}-1) \psi_h^{n+1}
-(|\psi^{n}|^{2}-1) \psi^{n+1},\varphi\big)  
+\big(i\eta\kappa{\bf A}_h^n,
\nabla((e_{\psi,h}^{n+1})^*\varphi)\big) 
 \nn\\
&\quad
-\big(i\eta\kappa{\bf A}_h^n,
\nabla((\psi^{n+1}-R_h\psi^{n+1})^*\varphi)\big) 
+\big(i\eta\kappa({\bf A}_h^n-{\bf A}^n),
\nabla((\psi^{n+1})^*\varphi)\big)  ,
\label{ErrFEMEqpsi}\\[10pt]
%%%%%%%%%%%%%%%%%%%%%%%%%%%
&(\nabla e_{p,h}^{n+1},\nabla \xi)\nn\\
&=
- \big(E_p^{n+1},\nabla\times \xi\big)
+\big({\rm Re}
[((\chi(\psi_h^n)^*-\chi(R_h\psi^n)^*
)(i\kappa^{-1}\nabla\psi^{n+1}
+ \mathbf{A}^n \psi^{n+1})],
\nabla\times \xi\big) \nn\\
&\quad\, -\big({\rm Re}
[(\chi(\psi^n)^*-\chi(R_h\psi^n)^*)(i\kappa^{-1}\nabla\psi^{n+1}
+ \mathbf{A}^n \psi^{n+1})],
\nabla\times \xi\big)\nn\\
&\quad\,  +\big({\rm Re}
[\chi(\psi_h^n)^*(i\kappa^{-1}\nabla e_{\psi,h}^{n+1} 
+ ({\bf A}_h^n-{\bf A}^n) \psi^{n+1} 
+{\bf A}_h^n e_{\psi,h}^{n+1} )],
\nabla\times \xi\big) \nn\\
&\quad\,  -\big({\rm Re}
[\chi(\psi_h^n)^*(i\kappa^{-1}\nabla(\psi^n-R_h\psi^n)
+{\bf A}_h^n (\psi^n-R_h\psi^n) )],
\nabla\times \xi\big) ,  \label{ErrFEMEqp}\\[10pt]
%%%%%%%%%%%%%%%%%%%%%%%%%%%
&(\nabla e_{q,h}^{n+1},\nabla\zeta)\nn\\
&=- \big(E_q^{n+1},\nabla\zeta\big)
+\big({\rm Re}
[((\chi(\psi_h^n)^*-\chi(R_h\psi^n)^*
)(i\kappa^{-1}\nabla\psi^{n+1}
+ \mathbf{A}^n \psi^{n+1})],
\nabla\zeta\big) \nn\\
&\quad\,
-\big({\rm Re}
[(\chi(\psi^n)^*-\chi(R_h\psi^n)^*)(i\kappa^{-1}\nabla\psi^{n+1}
+ \mathbf{A}^n \psi^{n+1})],
\nabla\zeta\big)\nn\\
&\quad\, +\big({\rm Re}
[\chi(\psi_h^n)^*(i\kappa^{-1}\nabla e_{\psi,h}^{n+1} 
+ ({\bf A}_h^n-{\bf A}^n) \psi^{n+1} 
+{\bf A}_h^n e_{\psi,h}^{n+1} ],
\nabla\zeta\big) \nn\\
&\quad\, -\big({\rm Re}
[\chi(\psi_h^n)^*(i\kappa^{-1}\nabla(\psi^n-R_h\psi^n)
+{\bf A}_h^n (\psi^n-R_h\psi^n) ],
\nabla\zeta\big) , \label{ErrFEMEqq} \\[10pt]
&\big(D_\tau e_{u,h}^{n+1}, \theta\big) 
+\big(\nabla e_{u,h}^{n+1},\nabla\theta\big) \nn\\
&=\big(D_\tau(u^{n+1}-\mathring R_hu^{n+1}), \theta\big) 
+\big( p^{n+1}-p_h^{n+1},\theta\big)
- \big(E_u^{n+1}, \theta\big) , \label{ErrFEMEqu}\\[10pt]
&\big(D_\tau e_{v,h}^{n+1}, \vartheta\big) 
+\big(\nabla e_{v,h}^{n+1},\nabla\vartheta\big)\nn\\
&=\big(D_\tau(v^{n+1}- R_hv^{n+1}), \vartheta\big) 
+( q^{n+1}-q_h^{n+1},\vartheta)
- \big(E_v^{n+1}, \vartheta\big) ,
\label{ErrFEMEqv}
\end{align}
for all $\varphi\in{\mathcal  V}^1_h$, $\xi,\theta\in \mathring V^1_h$
and $\zeta,\vartheta\in V^1_h$,
with  $\|e_{\psi,h}^0\|_{{\mathcal  L}^2}\leq Ch^{2s}$,
$\|e_{\psi,h}^0\|_{{\mathcal  H}^1}\leq Ch^{s}$
and $e_{u,h}^{0}=e_{v,h}^{0}=0$. 

Substituting $\theta=D_\tau e_{u,h}^{n+1}$
and $\vartheta=D_\tau e_{v,h}^{n+1}$ into 
\refe{ErrFEMEqu}-\refe{ErrFEMEqv}, we get
\begin{align*}
&\|D_\tau e_{u,h}^{n+1}\|_{L^2}^2 
+\|\Delta_h e_{u,h}^{n+1}\|_{L^2}^2 
+D_\tau \|\nabla e_{u,h}^{n+1}\|_{L^2}^2 \\
&\leq C\|D_\tau(u^{n+1}-\mathring R_hu^{n+1})\|_{L^2}^2+
C\|p^{n+1}-p_h^{n+1}\|_{L^2}^2
+C\|E_u^{n+1}\|_{L^2}^2,\\[8pt]
&\|D_\tau e_{v,h}^{n+1}\|_{L^2}^2 
+\|\Delta_h e_{v,h}^{n+1}\|_{L^2}^2 
+D_\tau \|\nabla e_{v,h}^{n+1}\|_{L^2}^2 \\
&\leq C\|D_\tau(v^{n+1}- R_hv^{n+1})\|_{L^2}^2+
 C\|q^{n+1}-q_h^{n+1}\|_{L^2}^2
+C\|E_v^{n+1}\|_{L^2}^2, 
\end{align*}
where $\Delta_h e_{u,h}^{n+1}$ and $\Delta_h e_{v,h}^{n+1}$
are defined in Lemma \ref{LemdisEmb}.
By Lemma \ref{LemdisEmb},
the last two inequalities imply that
\begin{align*}
&C^{-1}\|e_{u,h}^{n+1}\|_{W^{1,4}}^2 
+D_\tau \|\nabla e_{u,h}^{n+1}\|_{L^2}^2 \\
&\leq C\|p^{n+1}-p_h^{n+1}\|_{L^2}^2
+C\|D_\tau(u^{n+1}-\mathring R_hu^{n+1})\|_{L^2}^2
+C\|E_u^{n+1}\|_{L^2}^2,\\[8pt]
&C^{-1}\|e_{v,h}^{n+1}\|_{W^{1,4}}^2 
+D_\tau \|\nabla e_{v,h}^{n+1}\|_{L^2}^2 \\
&\leq C\|q^{n+1}-q_h^{n+1}\|_{L^2}^2
+C\|D_\tau(v^{n+1}- R_hv^{n+1})\|_{L^2}^2
+C\|E_v^{n+1}\|_{L^2}^2 .
\end{align*}
The sum of the last two inequalities gives
\begin{align}\label{ErAhL4}
&C^{-1}\|e_{{\bf A},h}^{n+1}\|_{L^{4}}^2 
+D_\tau \big(\|\nabla e_{u,h}^{n+1}\|_{L^2}^2
+\|\nabla e_{v,h}^{n+1}\|_{L^2}^2 \big) \nn \\
&\leq C\|p^{n+1}-p_h^{n+1}\|_{L^2}^2+C\|q^{n+1}-q_h^{n+1}\|_{L^2}^2
+C\|E_u^{n+1}\|_{L^2}^2+C\|E_v^{n+1}\|_{L^2}^2 \nn\\
&\quad +C\|D_\tau(u^{n+1}-\mathring R_hu^{n+1})\|_{L^2}^2
+C\|D_\tau(v^{n+1}- R_hv^{n+1})\|_{L^2}^2 .
\end{align}

At this moment, we invoke a mathematical induction on
\begin{align}\label{MindAsp}
&\|{\bf A}_h^n\|_{L^4}
\leq \max_{0\leq n\leq N}\|{\bf A}^n\|_{L^4}+1.
\end{align}
Since 
\begin{align*}
\|{\bf A}_h^0-{\bf A}^0\|_{L^4}
&\leq \|\nabla\times (\mathring R_hu^0-u^0)\|_{L^4}
+\|\nabla(R_hv^0-v^0)\|_{L^4}\\
&\leq Ch^{s-1/2}(\|u^0\|_{H^{1+s}}+\|v^0\|_{H^{1+s}}),
\end{align*}
there exists a positive constant $h_1$ such that 
\refe{MindAsp} holds for $n=0$ when $h<h_1$.
In the following, we present estimates of the finite element solution 
by assuming that \refe{MindAsp} holds for $0\leq n\leq m$, for some
nonnegative integer $m$.
We shall see that if \refe{MindAsp} holds for $0\leq n\leq m$,
then it also holds for $n=m+1$. 

Substituting $\xi=e_{p,h}^{n+1}$ in \refe{ErrFEMEqp},
it is not difficult to derive that
\begin{align}
\|\nabla e_{p,h}^{n+1}\|_{L^{2}} 
&\leq C\|E_p^{n+1}\|_{L^2} 
+C\|e_{\psi,h}^{n}\|_{L^4}\|i\kappa^{-1}\nabla\psi^{n+1}+{\bf A}^n\psi^{n+1}\|_{L^4} \nn \\
&\quad +C\|\psi^n-R_h\psi^n\|_{L^4}\|(i\kappa^{-1}\nabla\psi^{n+1}
+ \mathbf{A}^n \psi^{n+1})\|_{L^4}  \nn \\
&\quad +C\big(\|  \nabla e_{\psi,h}^{n+1} \|_{L^2}
+\| {\bf A}_h^n-{\bf A}^n \|_{L^2}
+\|{\bf A}_h^n\|_{L^4}\| e_{\psi,h}^{n+1} \|_{L^4} \nn  \\
&\quad +C\big(\|\nabla(\psi^{n+1}-R_h\psi^{n+1})\|_{L^2}
+\|{\bf A}_h^n\|_{L^4} \|\psi^n-R_h\psi^n\|_{L^4}\big) \nn \\
&\leq C\|E_p^{n+1}\|_{L^2}  +C\|e_{\psi,h}^{n}\|_{H^1}
+C\| \psi^{n+1}-R_h\psi^{n+1} \|_{H^1} 
 +C\| {\bf A}_h^n-{\bf A}^n \|_{L^2} . \label{ErrH1ep}
 \end{align}
Similarly, by substituting $\zeta=e_{q,h}^{n+1}$ in \refe{ErrFEMEqp},
one can derive that
\begin{align}\label{ErrH1eq}
\|\nabla e_{q,h}^{n+1}\|_{L^{2}} 
&\leq C\|E_q^{n+1}\|_{L^2}  +C\|e_{\psi,h}^{n}\|_{H^1}
+C\| \psi^n-R_h\psi^n \|_{H^1} 
 +C\| {\bf A}_h^n-{\bf A}^n \|_{L^2}  .
 \end{align}
Substituting $\varphi=e_{\psi,h}^{n+1}$ in \refe{ErrFEMEqpsi},
we obtain that
\begin{align}\label{errpsihL2}
& D_\tau\bigg(\frac{\eta}{2}  \|e_{\psi,h}^{n+1}\|_{L^2}^2\bigg) + 
\kappa^{-2} \|\nabla  
e_{\psi,h}^{n+1}\|_{L^2}^2  \nn\\
&\leq C\|D_\tau (\psi^{n+1}-R_h\psi^{n+1})\|_{L^2}^2
+C\|e_{\psi,h}^{n+1}\|_{L^2}^2+C\|E_\psi^{n+1}\|_{L^2}^2
  \nn\\
&\quad +C\|{\bf A}^{n}_h\|_{L^4}
\|\nabla e_{\psi,h}^{n+1}\|_{L^2}
\|e_{\psi,h}^{n+1}\|_{L^4} +C\|{\bf A}^{n}_h\|_{L^4}
\|\nabla (\psi_h^{n+1}-R_h\psi^{n+1})\|_{L^2}
\|e_{\psi,h}^{n+1}\|_{L^4} \nn \\
&\quad +C\|  {\bf A}_h^{n}-{\bf A}^{n}\|_{L^2}
\|\nabla \psi^{n+1}\|_{L^4} \|e_{\psi,h}^{n+1}\|_{L^4}
+C\|e_{\psi,h}^{n+1}\|_{L^4} \|{\bf A}_h^{n}\|_{L^4}
\|\nabla e_{\psi,h}^{n+1}\|_{L^2}   \nn\\
&\quad +C\|\psi^{n+1}
-R_h\psi^{n+1}\|_{L^4}\|{\bf A}_h^{n}\|_{L^4}
\|\nabla e_{\psi,h}^{n+1}\|_{L^2}
+C\|  {\bf A}_h^{n}-{\bf A}^{n}\|_{L^2}\|\nabla e_{\psi,h}^{n+1}\|_{L^2}  \nn\\ 
&\quad +C(\|{\bf A}_h^{n}\|_{L^4}+\|{\bf A}^{n}\|_{L^4})
\|  {\bf A}_h^{n}-{\bf A}^{n}\|_{L^2} \|e_{\psi,h}^{n+1}\|_{L^4} \nn\\
&\quad +C\|{\bf A}^n_h\|_{L^4}^2\|e_{\psi,h}^{n+1}\|_{L^4}^2
 +C\|{\bf A}^n_h\|_{L^4}^2\|\psi^{n+1}-R_h\psi^{n+1}\|_{L^4}\|e_{\psi,h}^{n+1}\|_{L^4} \nn\\
&\quad +(C\|\psi_h^{n}\|_{L^4}^2+C) \|e_{\psi,h}^{n+1}\|_{L^4}^2 
+(C\|\psi_h^{n}\|_{L^4} +C) \|e_{\psi,h}^{n}\|_{L^2}\|e_{\psi,h}^{n+1}\|_{L^4}  \nn\\
&\quad +C\|{\bf A}_h^n\|_{L^4}\|\nabla e_{\psi,h}^{n+1}\|_{L^2}
\|e_{\psi,h}^{n+1}\|_{L^4}
+C\|{\bf A}_h^n\|_{L^4}\|\psi^{n+1}-R_h\psi^{n+1}\|_{H^1}
\|e_{\psi,h}^{n+1}\|_{H^1}\nn\\
&\quad +C\|{\bf A}_h^n-{\bf A}^n\|_{L^2}
(\|e_{\psi,h}^{n+1}\|_{L^2}+\|\nabla e_{\psi,h}^{n+1}\|_{L^2})\nn\\
&\leq \epsilon\|\nabla e_{\psi,h}^{n+1}\|_{L^2}^2
+C_\epsilon\|e_{\psi,h}^{n+1}\|_{L^2}^2
+C_\epsilon\| {\bf A}_h^n-{\bf A}^n \|_{L^2}^2  \nn\\
&\quad +C \big(  \|E_\psi^{n+1}\|_{L^2}^2
+  \|D_\tau (\psi^{n+1}-R_h\psi^{n+1})\|_{L^2}^2 
+\|\psi_h^{n+1}-R_h\psi^{n+1}\|_{H^1}^2
 \big) ,
\end{align}
for any small positive number $\epsilon\in(0,1)$.
Substituting \refe{ErrH1ep}-\refe{ErrH1eq} into \refe{ErAhL4},
then \refe{ErAhL4} times $\varepsilon_1$
plus \refe{errpsihL2} gives
\begin{align} 
&\varepsilon_1 C^{-1}\|e_{{\bf A},h}^{n+1}\|_{L^{4}}^2 
+\kappa^{-2} \|\nabla   e_{\psi,h}^{n+1}\|_{L^2}^2
+D_\tau \big(\varepsilon_1 \|\nabla e_{u,h}^{n+1}\|_{L^2}^2
+\varepsilon_1 \|\nabla e_{v,h}^{n+1}\|_{L^2}^2
+\frac{\eta}{2}  \|e_{\psi,h}^{n+1}\|_{L^2}^2 \big)   \nn\\
&\leq C  \|E_p^{n+1}\|_{L^2}^2
+C \|E_q^{n+1}\|_{L^2}^2
+C\|E_u^{n+1}\|_{L^2}^2
+C\|E_v^{n+1}\|_{L^2}^2 
+C\|E_\psi^{n+1}\|_{L^2}^2  \nn  \\
&\quad +C\| \psi^{n+1}-R_h\psi^{n+1} \|_{H^1}^2 
+ C \|D_\tau (\psi^{n+1}-R_h\psi^{n+1})\|_{L^2}^2 
+C\|D_\tau (u^{n+1}-\mathring R_hu^{n+1})\|_{L^2}^2  \nn\\
&\quad 
+C\|D_\tau (v^{n+1}-R_hv^{n+1})\|_{L^2}^2
+(C\varepsilon_1+\epsilon)\|\nabla e_{\psi,h}^{n+1}\|_{L^2}^2 \nn\\
&\quad +(C\varepsilon_1+C_\epsilon)\|e_{\psi,h}^{n+1}\|_{L^2}^2
 + C_\epsilon \| {\bf A}_h^n-{\bf A}^n \|_{L^2}^2  .
 \label{prefnerr}
\end{align}
By choosing $\varepsilon_1$ and $\epsilon$ small enough, 
the term $(C\varepsilon_1+\epsilon)\|\nabla e_{\psi,h}^{n+1}\|_{L^2}^2$
on the right-hand side of the last inequality can be
eliminated by the left-hand side.
Since 
\begin{align*} 
\|{\bf A}_h^{n}-{\bf A}^{n}\|_{L^4}
&\leq 
C\|\nabla e_{u,h}^{n}\|_{L^2}+C\|\nabla e_{v,h}^{n}\|_{L^2}
+C\|\nabla\times  (u^{n}-\mathring R_hu^{n})\|_{L^4} 
+C\|\nabla (v^{n}-R_hv^{n})\|_{L^4}\\
&\leq C\|\nabla e_{u,h}^{n}\|_{L^2}+C\|\nabla e_{v,h}^{n}\|_{L^2} 
+ C(\|u^{n+1} \|_{H^{1+s}} +\|v^{n+1} \|_{H^{1+s}})h^s ,
\end{align*}
the inequality \refe{prefnerr} reduces to
\begin{align*} 
& \frac{\varepsilon_1}{ C }\|e_{{\bf A},h}^{n+1}\|_{L^{4}}^2 
+\frac{1}{2\kappa^2}\|\nabla   e_{\psi,h}^{n+1}\|_{L^2}^2
+D_\tau \big( \varepsilon_1 \|\nabla e_{u,h}^{n+1}\|_{L^2}^2
+\varepsilon_1  \|\nabla e_{v,h}^{n+1}\|_{L^2}^2
+\frac{\eta}{2}  \|e_{\psi,h}^{n+1}\|_{L^2}^2 \big)   \\
%&\leq  C\|\nabla e_{u,h}^{n}\|_{L^2}^2
% +C\|\nabla e_{v,h}^{n}\|_{L^2}^2
%+C \|e_{\psi,h}^{n+1}\|_{L^2}^2  \\
%&\quad 
%+C \|E_p^{n+1}\|_{L^2}^2+C \|E_q^{n+1}\|_{L^2}^2
%+C\|E_u^{n+1}\|_{L^2}^2
%+C\|E_v^{n+1}\|_{L^2}^2 
%+C\|E_\psi^{n+1}\|_{L^2}^2  \\
%&\quad 
%+C\| \psi^{n+1}-R_h\psi^{n+1} \|_{H^1}^2
%+C\| u^{n+1}-\mathring R_hu^{n+1} \|_{H^1}^2
%+C\| v^{n+1}-R_hv^{n+1} \|_{H^1}^2 \\
%&\quad + C \|D_\tau (\psi^{n+1}-R_h\psi^{n+1})\|_{L^2}^2 
%+C\|D_\tau (u^{n+1}-\mathring R_hu^{n+1})\|_{L^2}^2
%+C\|D_\tau (v^{n+1}- R_hv^{n+1})\|_{L^2}^2\\
&\leq C\|\nabla e_{u,h}^{n}\|_{L^2}^2+C\|\nabla e_{v,h}^{n}\|_{L^2}^2 
+C \|e_{\psi,h}^{n+1}\|_{L^2}^2\\
&\quad +C \|E_p^{n+1}\|_{L^2}^2+C \|E_q^{n+1}\|_{L^2}^2
+C\|E_u^{n+1}\|_{L^2}^2
+C\|E_v^{n+1}\|_{L^2}^2 
+C\|E_\psi^{n+1}\|_{L^2}^2  \\
&\quad + C (\|\psi^{n+1} \|_{H^{1+s}}^2
+\|u^{n+1} \|_{H^{1+s}}^2 
+\|v^{n+1} \|_{H^{1+s}}^2)h^{2s} \\
&\quad + C (\|D_\tau  \psi^{n+1} \|_{H^{1+s}}^2
+\|D_\tau u^{n+1} \|_{H^{1+s}}^2 
+\|D_\tau v^{n+1} \|_{H^{1+s}}^2)h^{4s}  .
\end{align*}
By applying Gronwall's inequality,
there exists a positive constant $\tau_1$ 
such that when $\tau<\tau_1$
we have 
\begin{align} \label{FNErrEst}
& \max_{0\leq n\leq m}\big( \|\nabla e_{u,h}^{n+1}\|_{L^2}^2
+ \|\nabla e_{v,h}^{n+1}\|_{L^2}^2
+ \|e_{\psi,h}^{n+1}\|_{L^2}^2 \big)  
+\sum_{n=0}^m\tau \|e_{{\bf A},h}^{n+1}\|_{L^{4}}^2   
\leq C_1(\tau^2+h^{2s})   
\end{align}
for some positive constant $C_1$.
In particular, the last inequality implies that
\begin{align*} 
 \max_{0\leq n\leq m} \|e_{{\bf A},h}^{n+1}\|_{L^2}^2 
 +\sum_{n=0}^m\tau \|e_{{\bf A},h}^{n+1}\|_{L^{4}}^2  
 \leq C(\tau^2+h^{2s})  .
\end{align*}
If $\tau\geq h $, then we have
\begin{align*} 
\|e_{{\bf A},h}^{m+1}\|_{L^{4}}^2  
\leq \frac{1}{\tau}\sum_{n=0}^m\tau \|e_{{\bf A},h}^{n+1}\|_{L^{4}}^2 
\leq C(\tau +h^{2s }/\tau) \leq C(\tau +h^{2s-1} ) \, ;
\end{align*}
if $\tau\leq h$, then we have
\begin{align*} 
\|e_{{\bf A},h}^{m+1}\|_{L^4}^2 
\leq Ch^{-1} \|e_{{\bf A},h}^{m+1}\|_{L^2}^2
\leq C(\tau^2 /h+h^{2s-1} ) \leq C(h +h^{2s-1} ) \, .
\end{align*}
Overall, we have $\|e_{{\bf A},h}^{m+1}\|_{L^4}^2 \leq 
C(\tau+h+h^{2s-1})$ and so
\begin{align*} 
\|{\bf A}_h^{m+1}-{\bf A}^{m+1}\|_{L^4}
&\leq\|e_{{\bf A},h}^{m+1}\|_{L^4}+
\|\nabla\times  (u^{m+1}-\mathring R_hu^{m+1})\|_{L^4} 
+\|\nabla (v^{m+1}-R_hv^{m+1})\|_{L^4}\\
%&\leq \|e_{{\bf A},h}^{m+1}\|_{L^4}+Ch^{s-1/2}\\
&\leq C(\tau^{1/2}+h^{1/2} +h^{s-1/2} )  .
\end{align*}
There exist positive constants $\tau_2$ and $h_2$ such
that when $\tau<\tau_2$ and $h<h_2$ we have
\begin{align*} 
\|{\bf A}_h^{m+1}-{\bf A}^{m+1}\|_{L^4}\leq 1 \, ,
\end{align*}
and this completes the mathematical induction on \refe{MindAsp}
in the case that $\tau<\tau_2$ and $h<h_2$.
Thus \refe{FNErrEst} holds for $m= N-1$ with the same constant
$C_1$, provided $\tau<\tau_2$ and $h<h_2$.

If $\tau\geq \tau_2$ or $h\geq h_2$, 
from \refe{APEPsi}-\refe{APEpq} we see that
\begin{align} \label{FNErrEst2}
& \max_{0\leq n\leq N-1}\big( \|\nabla e_{u,h}^{n+1}\|_{L^2}^2
+ \|\nabla e_{v,h}^{n+1}\|_{L^2}^2
+ \|e_{\psi,h}^{n+1}\|_{L^2}^2 \big)   
\leq C_2 \leq C_2\big(\tau_2^{-2}+h_2^{-2s}\big)(\tau^2+h^{2s})
\end{align}
for some positive constant $C_2$.
From \refe{FNErrEst} and \refe{FNErrEst2} we see that 
for any $\tau$ and $h$ we have
\begin{align*}  
&  \max_{0\leq n\leq N-1}\big( \|\nabla e_{u,h}^{n+1}\|_{L^2}^2
+ \|\nabla e_{v,h}^{n+1}\|_{L^2}^2
+ \|e_{\psi,h}^{n+1}\|_{L^2}^2 \big)  
\leq \big[C_1+C_2\big(\tau_2^{-2}+h_2^{-2s}\big)\big](\tau^2+h^{2s}) .
\end{align*}

The proof of Theorem \ref{MainTHM3} is completed. 
\bigskip

\section{Numerical example}\label{NumerEx}
\setcounter{equation}{0}

%In this section, we present numerical examples 
%to support our theoretical analysis.
%In particular, we compare the numerical solution 
%of the TDGL 
%with the numerical solution 
%of the projected TDGL.

%\begin{example} 
We consider an artificial example, the equations 
\begin{align}
&\eta\frac{\partial \psi}{\partial t} 
+ \bigg(\frac{i}{\kappa} \nabla 
+ \mathbf{A}\bigg)^{2} \psi
 + (|\psi|^{2}-1) \psi -i\eta \kappa \psi \nabla\cdot{\bf A} = g ,
\label{NTPDE1}\\[5pt]
&\frac{\partial \mathbf{A}}{\partial t} 
+ \nabla\times(\nabla\times{\bf A})
-\nabla(\nabla\cdot{\bf A}) 
+  {\rm Re}\bigg[\psi^*\bigg(\frac{i}{\kappa} \nabla 
+ \mathbf{A}\bigg) \psi\bigg] 
=  {\bf g}+\nabla\times f ,
\label{NTPDE2}
\end{align}
in an L-shape domain $\Omega$ 
whose longest side has unit length, centered at the origin,
with $\eta=1$ and $k=10$.
The functions 
$f=\nabla\times{\bf A}\in C^1([0,T];{\bf H}^2)$, 
$g\in C([0,T];L^2)$ and 
${\bf g}\in C([0,T];{\bf L}^2)$ are chosen
corresponding to the exact solution
\begin{align*}
&\psi=t^2\Phi(r)r^{2/3}\cos(2\theta/3),\\
&{\bf A}=\Big (\big( 4t^2 \Phi(r)r^{-1/3}/3
+t^2\Phi'(r)r^{2/3}\big)\cos(\theta/3),~
\big( 4t^2 \Phi(r)r^{-1/3}/3
+t^2\Phi'(r)r^{2/3}\big)\sin(\theta/3)\Big ),
\end{align*}
where $(r,\theta)$ denotes the polar coordinates, 
the cut-off function $\Phi(r)$ is defined by
$$
\Phi(r)=\left\{
\begin{array}{ll}
0.1 & \mbox{if}~~r<0.1,  \\
\Upsilon(r) &\mbox{if}~~ 0.1\leq r\leq 0.4 ,\\
0 & \mbox{if}~~r>0.4, 
\end{array}\right.
$$
and $\Upsilon(r)$ is the unique $7^{\rm th}$ 
order polynomial satisfying the 
conditions $\Upsilon'(0.1)=\Upsilon''(0.1)
=\Upsilon'''(0.1)=\Upsilon(0.4)=\Upsilon'(0.4)
=\Upsilon''(0.4)=\Upsilon'''(0.4)=0$ 
and $\Upsilon(0.1)=0.1$.
It is easy to check that the exact solution $(\psi,{\bf A})$
satisfies the boundary and initial conditions
\refe{bc}-\refe{init} with $\psi_0=0$ and ${\bf A}_0=(0,0)$.

The L-shape domain is triangulated quasi-uniformly, 
as shown in Figure \ref{Lshape}, 
with $M$ nodes per unit length on each side, 
and we denote by $h=1/M$ for simplicity.
\begin{figure}[ht]
\vspace{0.1in}
\centering
\begin{tabular}{ccc}
\epsfig{file=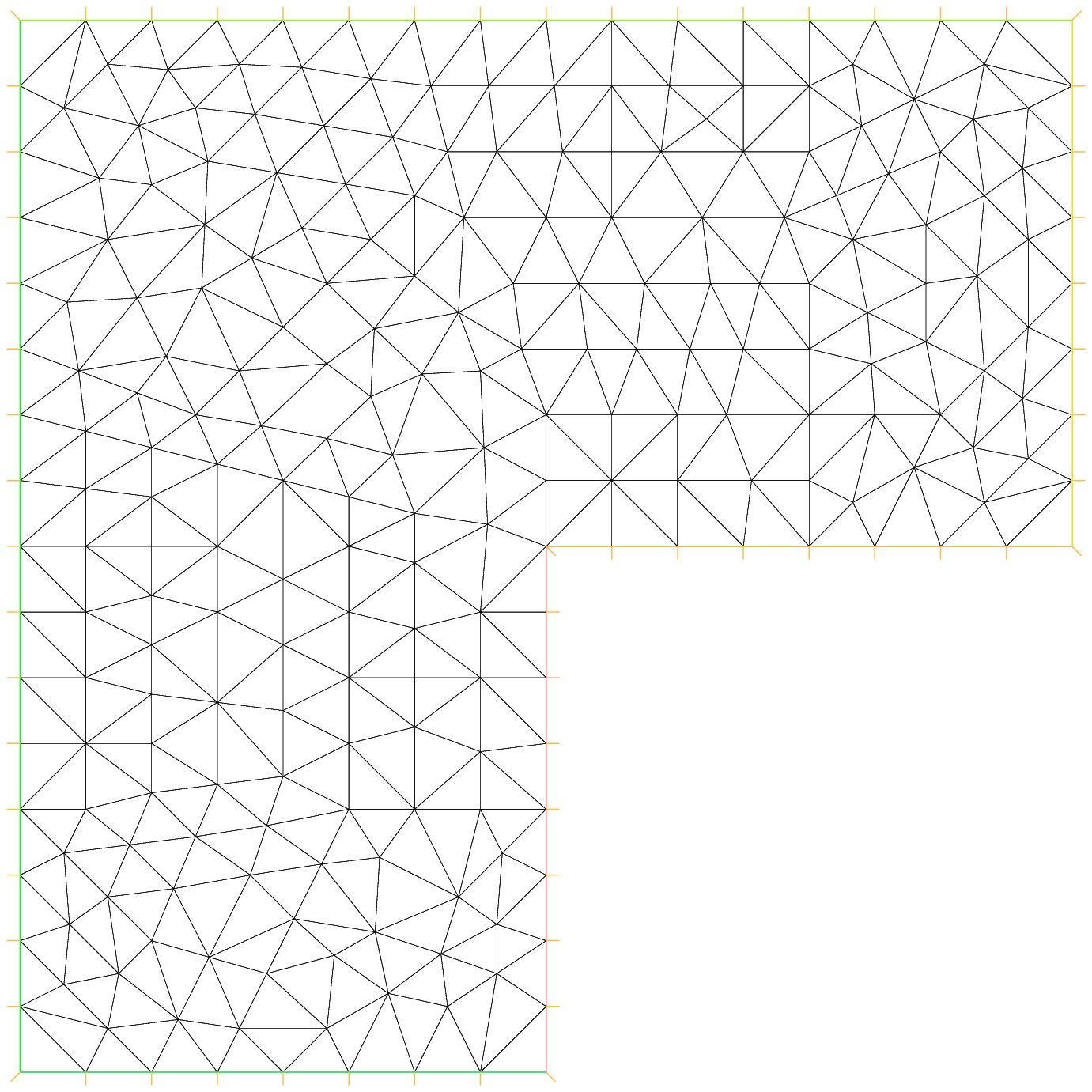,height=1.3in,width=1.9in}
\hspace{-0.2in}
\epsfig{file=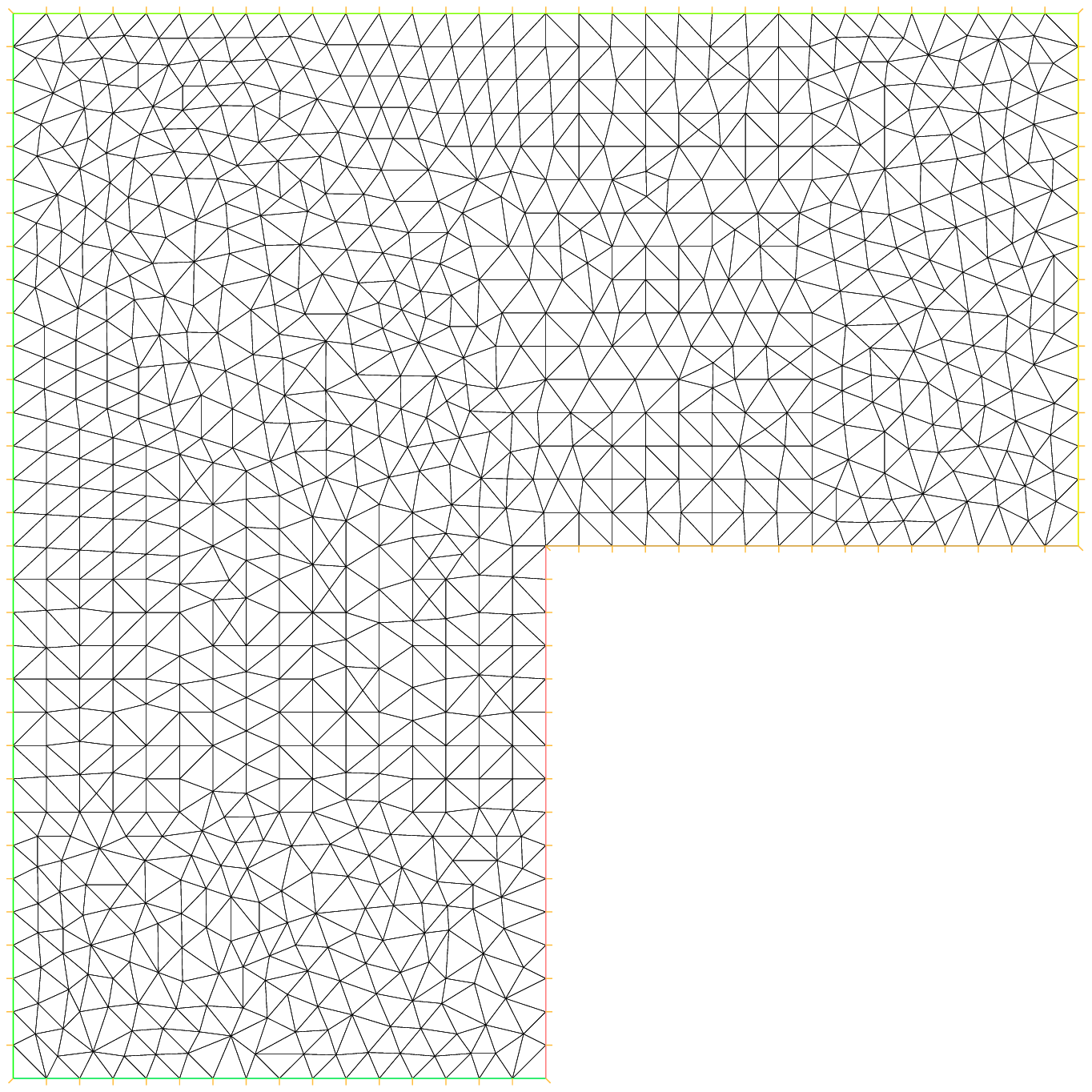,height=1.3in,width=1.9in}
\hspace{-0.2in}
\epsfig{file=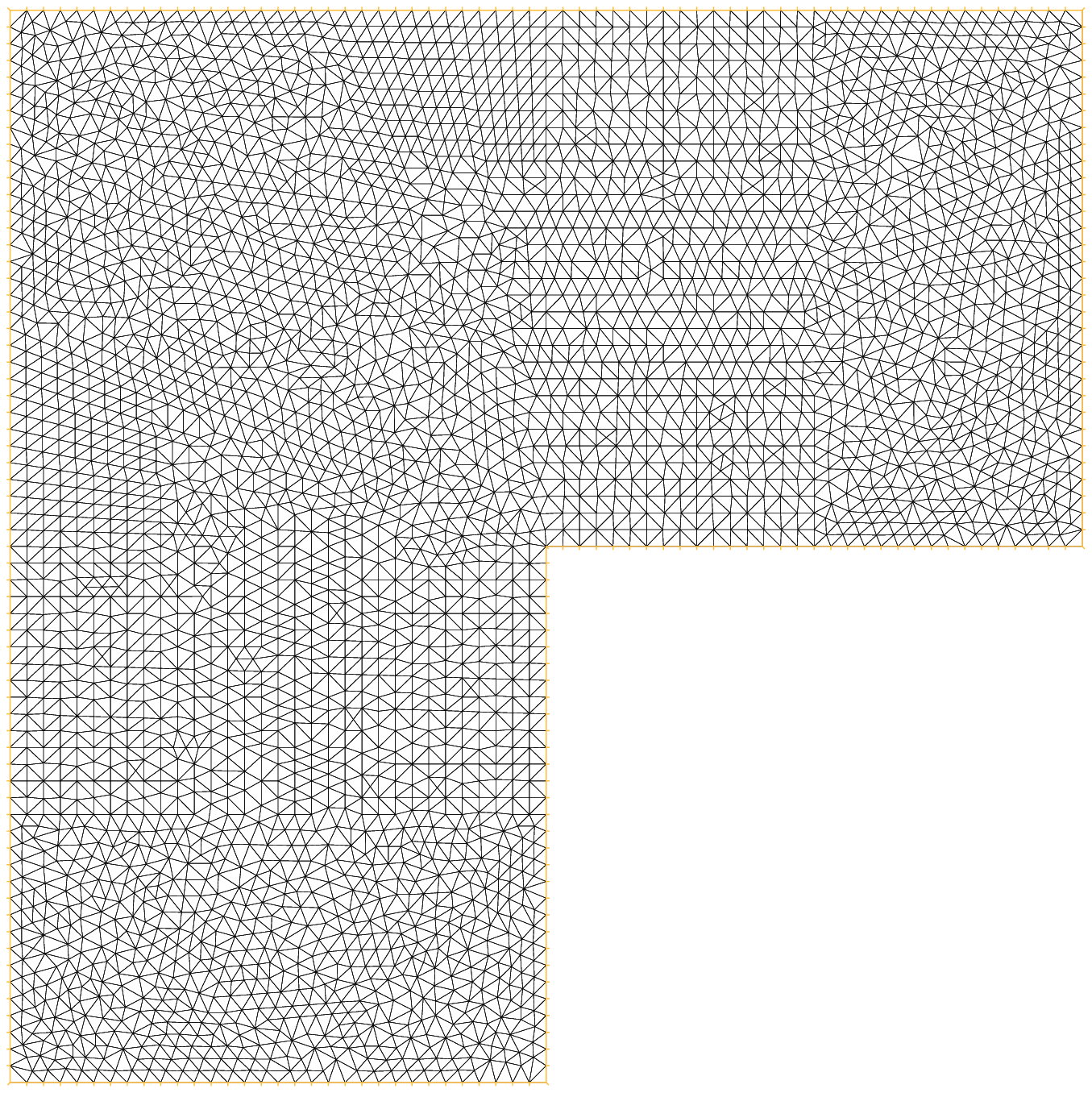,height=1.3in,width=1.9in}
\end{tabular}
%\vskip-0.1in
\caption{\small Quasi-uniform triangulations with $M=16,32,64$.}
\label{Lshape}
\end{figure}

Firstly, we solve \refe{NTPDE1}-\refe{NTPDE2} directly 
by the FEM with piecewise linear finite elements 
and a linearized backward Euler scheme, 
%(see \cite{CH95,DuFEM94, GLS,Mu97,MH98} concerning 
%FEMs for the TDGL),
and we denote the numerical solution by 
$(\widetilde\psi_h^N,\widetilde{\bf A}_h^N)$. 
%Specifically, 
%we look for $\widetilde\psi^{n+1}_h\in {\mathcal  V}_{h}^{1}$ 
%and $\widetilde {\bf A}^{n+1}_h\in
%{\bf V}_h^1:=\{{\bf a}\in {\bf H}^1_{\rm n}(\Omega): {\bf a}|_K~ 
%\mbox{is linear for each triangle $K$, and}
%~{\bf a}\cdot{\bf n}=0~\mbox{on}~\partial\Omega\}$, $n=0,1,\cdots,N-1$, satisfying
%\begin{align}
%&\big(D_\tau \widetilde \psi_h^{n+1}, \varphi\big) +\big
%((i\kappa^{-1}\nabla +\widetilde {\bf A}^{n}_h)
%\widetilde \psi^{n+1}_h,(i\kappa^{-1}\nabla +\widetilde {\bf
%A}^{n}_h)  \varphi\big) +\big
%((|\widetilde \psi^{n}_h|^{2}-1) \widetilde \psi^{n+1}_h,
%\varphi\big) =\big(g,\varphi\big), 
%\label{C0FEM1}\\[5pt]
%&\big(D_\tau\widetilde {\bf A}_h^{n+1},{\bf a}\big)
%+\big(\nabla\times\widetilde {\bf A}_h^{n+1},\nabla\times{\bf a}\big)
%+\big(\nabla\cdot\widetilde {\bf A}_h^{n+1},\nabla\cdot{\bf a}\big) \nn\\
%&\qquad\qquad\qquad\quad~~ + \big( {\rm Re}\big[(\widetilde \psi_h^{n})^*\big(i\kappa^{-1}\nabla 
%+\widetilde{\bf A}_h^{n}\big) \widetilde \psi_h^{n+1}\big],{\bf a}\big) 
%= \big({\bf g},{\bf a}\big)+ \big( f ,\nabla\times{\bf a}\big),
%\label{C0FEM2}
%\end{align}
%for any $\varphi\in{\mathcal  V}_h^1$ and 
%${\bf a}\in {\bf V}_h^1$, with the initial condition $\widetilde\psi_h^0=0$
%and $\widetilde{\bf A}_h^0=(0,0)$. 
In a convex or smooth domain, convergence of 
the numerical solution 
$(\widetilde\psi_h^N,\widetilde{\bf A}_h^N)$ 
can be proved based on the method of \cite{CH95,DuFEM94}. 
Here we are interested in the question: whether 
the numerical solution converges to the 
correct solution in a nonconvex polygonal domain?
To answer this question, 
we present the errors of the numerical solution 
in Table \ref{Tab1} with $\tau=h$ for several different $h$.
One can see that the errors do not decrease as the mesh is refined.
In other words, the numerical solution 
$(\widetilde\psi_h^N,\widetilde{\bf A}_h^N)$ does not converge 
to the correct solution, nor does the physical quantity 
$|\widetilde\psi_h^N|$ converge to $|\psi^N|$. 
%This indicates that the direct Galerkin FEMs 
%with $C^0$ finite elements are, indeed, not suitable to solve
%the Ginzburg--Landau equations in a domain with reentrant corners.

Secondly, we solve the projected TDGL
corresponding to  \refe{NTPDE1}-\refe{NTPDE2}
by the proposed method
%by incorporating $g$ and ${\bf g}$ into the equations.
%Specifically, we look for $\psi^{n+1}_h\in {\mathcal  V}_{h}^{1}$, 
%$p^{n+1}_h,u^{n+1}_h\in
%\mathring V_h^{1}$ and
%$q^{n+1}_h,v^{n+1}_h\in
%V_h^{1}$ satisfying\begin{align*}
%&\big(D_\tau \psi^{n+1}_h, \varphi\big) +\big
%((i\kappa^{-1}\nabla + \mathbf{A}^{n}_h)
%\psi^{n+1}_h,(i\kappa^{-1}\nabla + \mathbf{
%A}^{n}_h)  \varphi\big) +\big
%((|\psi^{n}_h|^{2}-1) \psi^{n+1}_h,
%\varphi\big) = \big(g^{n+1},\varphi\big),&
%\\
%&(\nabla p^{n+1}_h,\nabla \xi)=\big({\rm Re}
%[\chi(\psi^n_h)^*(i\kappa^{-1}\nabla\psi^{n+1}_h 
%+ \mathbf{A}^n_h \psi^{n+1}_h)]-{\bf g}^{n+1},
%\nabla\times \xi\big)
%\\
%&(\nabla q^{n+1}_h,\nabla \zeta)
%=\big({\rm Re}
%[\chi(\psi^n_h)^*(i\kappa^{-1}\nabla\psi^{n+1}_h 
%+ \mathbf{A}^n_h \psi^{n+1}_h)]-{\bf g}^{n+1},\nabla \zeta\big)
%\\
%&\big(D_\tau u^{n+1}_h, \theta\big) 
%+\big(\nabla u^{n+1}_h,\nabla\theta\big)=(f^{n+1}-p^{n+1}_h,\theta)\\
%&\big(D_\tau v^{n+1}_h, \vartheta\big) 
%+\big(\nabla v^{n+1}_h,\nabla\vartheta\big)
%=(-q^{n+1}_h,\vartheta), 
%\end{align*}
%for all $\varphi\in{\mathcal  V}^1_h$, $\zeta,\vartheta\in V^1_h$ 
%and $\xi,\theta\in \mathring V^1_h$,
%with ${\bf A}_h^n=\nabla\times u_h^n+\nabla v_h^n$
%and $\psi_h^0=u^0_h=v^0_h=0$. 
and denote the numerical solution by 
$(\psi_h^N,{\bf A}_h^N)$. We present the errors of the 
numerical solution in Table \ref{Tab2},
where the convergence rate of $\psi_h^N$  
is calculated by the formula 
\begin{align*}
&{\rm convergence~ rate~ of~}\psi_h^N=
\log(\|\psi_h^N -\psi^N \|_{L^2}/\|\psi_{h/2}^N -\psi^N \|_{L^2})/\log 2 
%\\
%&{\rm convergence~ rate~ of~}|\psi_h^N|=
%\log(\||\psi_h^N| -|\psi^N| \|_{L^2}/
%\||\psi_{h/2}^N| -|\psi^N| \|_{L^2})/\log(2),\\
%&{\rm convergence~ rate~ of~}{\bf A}_h^N=
%\log(\|{\bf A}_h^N -{\bf A}^N \|_{L^2}/
%\|{\bf A}_{h/2}^N -{\bf A}^N \|_{L^2})/\log(2),\\
%&{\rm convergence~ rate~ of~}B_h^N=
%\log(\|B_h^N -B^N \|_{L^2}/\|B_{h/2}^N- B^N \|_{L^2})/\log(2),
\end{align*}
based on the finest mesh size $h$ 
(the same formula is used for $|\psi_h^N|$ and ${\bf A}_{h}^N$).
We see that the convergence rates of
$\psi_h^N$, $|\psi_h^N|$ and ${\bf A}_h^N$ are better than $O(h^{2/3})$, 
which is the worst convergence rate proved in Theorem \ref{MainTHM3}.
The numerical results are consistent with our theoretical analysis 
and indicate that our method is efficient for solving
the Ginzburg--Landau equations 
in a domain with reentrant corners.
%\end{example}

\begin{table}[htp]
\begin{center}
\caption{\small Errors of the finite 
element solution $(\widetilde\psi_h^N,\widetilde{\bf A}_h^N)$
with $\tau=h$.}
\label{Tab1}
\begin{tabular}{c|c|c|ccc}
\hline
 $h$ &  $\| \widetilde\psi_h^N -\psi^N \|_{L^2}$ 
 &  $\| |\widetilde\psi_h^N| -|\psi^N| \|_{L^2}$ &
$\|  \widetilde{\bf A}_h^N - {\bf A}^N \|_{L^2}$\\ 
%& $\|  B_h^N - B^N \|_{L^2}$        \\
\hline
  1/16
&4.2113E-03  &3.7007E-03  &8.3961E-02\\% &4.3434E-01  \\
  1/32
&3.1847E-03  &2.0651E-03 &8.1396E-02\\% &2.5710E-01\\
 1/64
&2.9884E-03  &1.6286E-03 &7.9709E-02\\% &1.7884E-01 \\
1/128
&2.9170E-03  &1.4624E-03 &7.8779E-02\\%  &1.5062E-01 \\
1/256
&2.8734E-03  &1.3875E-03 &7.8210E-02\\%  &1.4299E-01 \\
\hline convergence rate
&   $O(h^{0.02})$  &  $O(h^{0.07})$   &  $O(h^{0.01})$\\
%&   $O(h^{0.25})$  \\
\hline 
\end{tabular}
\end{center}
\bigskip\bigskip

\begin{center}
\caption{\small Errors of the finite 
element solution $(\psi_h^N,{\bf A}_h^N)$ with $\tau=h$.}
\label{Tab2}
\begin{tabular}{c|c|c|ccc}
\hline
 $h$ &  $\| \psi_h^N -\psi^N \|_{L^2}$ 
 &  $\| |\psi_h^N| -|\psi^N| \|_{L^2}$  &
$\|  {\bf A}_h^N - {\bf A}^N \|_{L^2}$ \\
%&$\|  B_h^N - B^N \|_{L^2}$     \\
\hline
  1/16
&2.7608E-03  &2.4889E-03 &2.9448E-02\\% & 3.2125E-03 \\
  1/32
&8.0517E-04  &7.0163E-04 &1.4861E-02\\% &9.2255E-04\\
 1/64
&3.1147E-04  &2.8685E-04 &8.0870E-03\\% & 1.9654E-04\\
1/128
&1.3066E-04  &1.2664E-04 &4.3397E-03\\% & 3.9270E-05\\
1/256
&6.1047E-05  &6.0252E-05 &2.3748E-03\\% & 7.3511E-06\\
\hline
convergence rate
& 
$O(h^{1.09})$   
& 
$O(h^{1.07})$ 
&  $O(h^{0.87})$  \\
%&  $O(h^{2.30})$  \\
\hline 
\end{tabular}

\end{center}
\end{table}

\section{Conclusions}\label{Concls}
\setcounter{equation}{0}
We have proved the well-posedness of the time-dependent Ginzburg--Landau
superconductivity model in a nonconvex polygonal domain.
Due to the singularity of the magnetic potential, direct application of the finite element method
to the original Ginzburg--Landau equations may yield an incorrect solution.
Based on the Hodge decomposition,
we reformulated the equations into an equivalent system,
 which avoids direct calculation of the magnetic potential, and therefore
can be solved correctly by finite element methods. Then a
decoupled and linearized FEM was proposed 
and convergence rate of the numerical solution was established
based on proved regularity of the essential unknowns of the reformulated system.
Numerical examples show the effectiveness of the proposed method
in comparison with the traditional approach.
For simplicity, we have focused on nonconvex polygons in this paper.
Nevertheless, the results 
can be extended to nonconvex curved polygons
without essential change of the argument.

\bigskip

\noindent{\bf Acknowledgement.}$\quad$
We would like to thank Professor Qiang Du
for helpful discussions.\bigskip

\end{document}